\documentclass{amsart}
\usepackage[mathcal]{eucal}
\usepackage{amssymb, labelfig, graphics}

\newtheorem{thm}{Theorem}
\newtheorem{prop}[thm]{Proposition}
\newtheorem{lem}[thm]{Lemma}

\theoremstyle{remark}
\newtheorem{rem}[thm]{Remark}

\theoremstyle{definition}

\newcommand{\R}{\mathbb R}
\newcommand{\Z}{\mathbb Z}

\newcommand{\HH}{\mathbb H}
\newcommand{\SL}{\mathrm{SL_n(\R)}}
\newcommand{\PSL}{\mathrm{PSL_n(\R)}}
\newcommand{\GL}{\mathrm{GL_n(\R)}}
\newcommand{\PGL}{\mathrm{PGL_n(\R)}}
\newcommand{\Flag}{\mathrm{Flag}(\R^n)}

\newcommand{\Right}{\mathrm{right}}
\newcommand{\Left}{\mathrm{left}}

\newcommand{\Hit}{\mathrm{Hit}_{\mathrm{n}}}
\newcommand{\Id}{\mathrm{Id}}

\newcommand{\db}{/\kern -3pt/}

\renewcommand{\leq}{\leqslant}
\renewcommand{\geq}{\geqslant}
\renewcommand{\phi}{\varphi}
\renewcommand{\epsilon}{\varepsilon}

\title
{Parametrizing Hitchin components}

\author{Francis Bonahon}
\address {Department
of Mathematics,  University of
Southern California, Los Angeles 
CA~90089-2532, U.S.A.}
\email{fbonahon@math.usc.edu}

\author{Guillaume Dreyer}
\address {Department
of Mathematics,  University of
Notre Dame, Notre Dame IN~46556, U.S.A.}
\email{dreyfactor@gmail.com, gdreyer@nd.edu}

\thanks{This research was partially supported by the grants  DMS-0604866 and DMS-1105402 from the National Science Foundation.}
\date{\today}

\begin{document}

\begin{abstract}
We construct a geometric, real analytic parametrization of the Hitchin component $\mathrm{Hit_n}(S)$ of the $\PSL$--character variety $\mathcal R_{\PSL}(S)$ of a closed surface $S$. The approach is explicit and constructive. In essence, our parametrization is an extension of Thurston's shearing coordinates for the Teichm\"uller space of a closed surface, combined with Fock-Goncharov's coordinates for the moduli space of positive framed local systems of a punctured surface. More precisely, given a maximal geodesic lamination $\lambda\subset S$ with finitely many leaves, we introduce two types of invariants for elements of the Hitchin component: shear invariants associated with each leaf of $\lambda$; and  triangle invariants associated with each component of the complement $S-\lambda$. We describe  identities and relations satisfied by these  invariants, and use the resulting coordinates to parametrize the Hitchin component.
\end{abstract}

\maketitle

For a closed, connected, oriented surface $S$ of genus $g>1$, this article is concerned with the space of homomorphisms $\rho \colon \pi_1(S) \to \PSL$ from the fundamental group $\pi_1(S)$ to the Lie group $\PSL$ (equal to the special linear group $\SL$ if $n$ is odd, and to  $\SL/\pm \Id$ if $n$ is even), and  more precisely with a preferred component $\Hit(S)$ of the character variety
$$
\mathcal R_{\PSL}(S) = \{ \text{homomorphisms } \rho \colon \pi_1(S) \to \PSL \} \db \PSL,
$$
where the group $\PSL$ acts on homomorphisms $\pi_1(S) \to \PSL$ by conjugation. The precise definition of the character variety $\mathcal R_{\PSL}(S) $ requires that the quotient be taken in the sense of geometric invariant theory \cite{Mum}; however, for the component $\Hit(S)$ that we are interested in, this quotient construction coincides with the usual topological quotient \cite{Hit}. Also, note that the consideration of homomorphisms $\pi_1(S) \to \PSL$ is essentially equivalent,  by arguments involving the cohomology groups $H^1(S;\R^*)$ and $H^1(S; \Z_2)$, to the analysis of general representations $\pi_1(S) \to \GL$. 

In the case where $n=2$, the character variety $\mathcal R_{\mathrm{PSL}_2(\R)}(S) $ has $4g-3$ components  \cite{Gold1}. Two of these components correspond to all injective homomorphisms $\rho\colon \pi_1(S) \to \mathrm{PSL}_2(\R)$ whose image $\rho \bigl( \pi_1(S) \bigr)$ is discrete in $\mathrm{PSL}_2(\R)$. Identifying $\mathrm{PSL}_2(\R)$ with the orientation-preserving isometry group of the hyperbolic plane $\HH^2$, the orientation of $S$ then singles out one of these two components, where the natural equivalence relation $S \to \HH^2/\rho \bigl( \pi_1(S) \bigr)$ has degree $+1$. This preferred component of $\mathcal R_{\mathrm{PSL}_2(\R)}(S) $ is the \emph{Teichm\"uller component} $\mathcal T(S)$. By the Uniformization Theorem, the Teichm\"uller component $\mathcal T(S)$ is diffeomorphic to the space of complex structures on $S$, and consequently plays a fundamental r\^ole in complex analysis as well as in $2$--dimensional hyperbolic geometry. In particular, a classical result is that it is diffeomorphic to $\R^{6(g-1)}$. 

In the general case, there is a preferred homomorphism $\mathrm{PSL}_2(\R) \to \PSL$ coming from the unique $n$--dimensional representation of $\mathrm{SL}_2(\R)$ (or, equivalently, from the natural action of $\mathrm{SL}_2(\R) $ on the vector space $\R[X,Y]_{n-1}\cong \R^n$ of homogeneous polynomials  of degree $n-1$ in two variables). This provides a natural map $\mathcal R_{\mathrm{PSL}_2(\R)}(S) \to \mathcal R_{\PSL}(S) $. The \emph{Hitchin component} $\Hit(S)$ is the component of $\mathcal R_{\PSL}(S)$ that contains the image of the Teichm\"uller component of $\mathcal R_{\mathrm{PSL}_2(\R)}(S) $. A \emph{Hitchin representation} is a homomorphism $\rho \colon \pi_1(S) \to \PSL$ representing an element of the Hitchin component. The terminology is motivated by the following fundamental result of N.~Hitchin \cite{Hit}, who was the first to single out this component. 

\begin{thm} [Hitchin]
\label{thm:Hitchin}
When $n\geq 3$, the character variety has $3$ or $6$ components according to whether $n$ is odd or even, and the Hitchin component $\Hit(S)$ is diffeomorphic to $\R^{2(g-1)(n^2-1)}$. 
\end{thm}

Hitchin's proof of these results is based on the theory of Higgs bundles, and relies on techniques  of geometric analysis. Hitchin notes in \cite{Hit} that these methods do not provide any geometric information on individual Hitchin representations. 

In the case $n=3$, S. Choi and W. Goldman \cite{ChGold} showed that the Hitchin component $\mathrm{{Hit}_{3}}(S)$ parametrizes the deformation space of {real convex projective structures} on $S$, and Goldman \cite{Gold2} used this point of view to construct   an explicit parametrization of  $\mathrm{{Hit}_{3}}(S)$ by $\R^{16(g-1)}$, via an extension of the classical Fenchel-Nielsen coordinates for the Teichm\"uller space $\mathcal T(S)$.


A decade later, F.~Labourie \cite{Lab1} showed in the general case  that Hitchin representations are injective and have discrete image in $\PSL$. He achieved this by establishing a very powerful Anosov property for  Hitchin representations. This Anosov property associates to each Hitchin representation $\rho\colon \pi_1(S) \to \PSL$ a  certain \emph{flag curve} valued in the space $\Flag$ of all flags in $\R^n$, which is invariant under the image  $\rho\bigl(\pi_1(S) \bigr) \subset \PSL$. The same invariant flag curve was similarly provided by independent work of V.~Fock and A.~Goncharov \cite{FoG1}, who in addition established a certain positivity property for this flag curve. This approach also proves the faithfulness and the discreteness of Hitchin representations. The point of view of  Fock and Goncharov is algebraic geometric and relies on  G.~Lusztig's notion of positivity \cite{Lusz1, Lusz2}; in particular, it  is very different from Labourie's.

The main achievement of the current paper is to provide a new parametrization of the Hitchin component $\Hit(S)$ by $\R^{2(g-1)(n^2-1)}$,  which is much more closely related to the geometry of Hitchin representations than Hitchin's original parametrization. It relies on the methods developed by Labourie and Fock-Goncharov.  When $n=2$, this parametrization coincides with the parametrization of the Teichm\"uller space $\mathcal T(S)$ via the shear coordinates associated with a maximal geodesic lamination $\lambda$ that were developed in \cite{Thu, Bon}. 

\begin{figure}[htbp]

\SetLabels
( .25*-.2 ) (a) \\
(.85 *-.2 ) (b) \\
\endSetLabels
\centerline{\AffixLabels{ \includegraphics{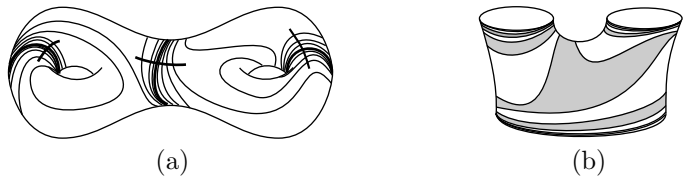} }}

\vskip 10pt

\caption{A finite geodesic lamination coming from a pair of pants decomposition}
\label{fig:GeodLam1}
\end{figure}

We begin with some topological data on the surface $S$, consisting of a maximal geodesic lamination $\lambda$ with finitely many leaves. For instance, in the spirit of the classical Fenchel-Nielsen coordinates (see for instance \cite[\S7.6]{Hub}) for the Teichm\"uller space $\mathcal T(S)$, such a geodesic lamination can be obtained from a decomposition of $S$ into pairs of pants,  by cutting each pair of pants along three bi-infinite curves spiraling around the boundary components to obtain three infinite triangles; the geodesic lamination $\lambda$ then consists of the $3(g-1)$ closed curves of the pair of pants decomposition, together with the $6(g-1)$ spiraling bi-infinite curves; see Figure~\ref{fig:GeodLam1}(a) for an example, while Figure~\ref{fig:GeodLam1}(b) illustrates how to split a pair of pants along three spiraling curves to obtain two infinite triangles. In general, for an arbitrary auxiliary metric of negative curvature on $S$, a maximal geodesic lamination with finitely many leaves consists of $s$ disjoint closed geodesics, with $1\leq s \leq 3(g-1)$, and of $6(g-1)$ disjoint bi-infinite geodesics whose ends spiral around these closed geodesics and which split the surface $S$ into $4(g-1)$ infinite triangles. 

Given a Hitchin representation  $\rho \colon \pi_1(S) \to \PSL$, the Anosov structure discovered by Labourie and the positivity property introduced by Fock-Goncharov enable us to read a certain number of invariants of $\rho$. These include $\frac12(n-1)(n-2)$ real numbers (called \emph{triangle invariants}) associated with each of the $4(g-1)$ triangles of $S-\lambda$, and $n-1$ real numbers (called \emph{shear invariants}) associated with each of the $6(g-1) +s$  leaves of the geodesic lamination $\lambda$. 

The triangle invariants, and the shear invariants associated with the infinite leaves, were introduced by Fock and Goncharov in their parametrization \cite{FoG1} of the so-called moduli space of positive framed local systems of a surface $S$, where $S$ is required to have at least one puncture. This moduli space is the natural extension of the Hitchin component to punctured surfaces; see \cite{BAG} for the Higgs bundle point of view on this space. The construction of the shear invariants associated with closed leaves is new, but very similar to that of infinite leaves.

A major difference with the punctured-surface case of Fock and Goncharov lies in the fact that,  when the surface $S$ is closed, the triangle and  shear invariants are not independent of each other. Indeed, they satisfy $n-1$ linear equalities and $n-1$ linear inequalities for each of the $s\geq 1$ closed leaves of $\lambda$. It turns out that these equalities and inequalities are the only relations satisfied by these invariants, and that they can be used to parametrize $\Hit(S)$.

\begin{thm}
\label{thm:Parametrize}
The above triangle and shear invariants provide a real-analytic pa\-ram\-e\-trization of the Hitchin component $\Hit(S)$ by the interior of a convex polytope of dimension $2(g-1)(n^2-1)$. 
\end{thm}

When $n=2$, there are no triangle invariants and, as indicated earlier, the parametrization of Theorem~\ref{thm:Parametrize} coincides with the now classical parametrization of the Teichm\"uller space $\mathcal T(S)$ by shear coordinates \cite{Thu, Bon}. For general $n$, because $S-\lambda$ consists of $4(g-1)$ triangles, the triangle invariants define a map $\Hit(S) \to \R^{2(g-1)(n-1)(n-2)}$. It turns out that there are global linear relations between these triangle invariants:


\begin{prop}
\label{prop:RelationsTriangleInvariants}
The image of the map  $\Hit(S) \to \R^{2(g-1)(n-1)(n-2)}$ defined by triangle invariants is contained in a linear subspace of codimension  $\lfloor \frac12(n-1) \rfloor$ in $\R^{2(g-1)(n-1)(n-2)}$ (where $\lfloor x \rfloor$ denotes the largest integer  $\leq x$). 
\end{prop}

The existence of constraints for the triangle invariants was somewhat unexpected to us. They can be explained by a more conceptual approach that uses the length functions of \cite{Dre}, combined with a homological argument; see  \cite{BonDre}. In fact, the abstract proof of \cite{BonDre} preceded the explicit computational argument that we give in the current article.

\section{Generic triples and quadruples of flags}
\label{sect:GenericFlagTriplesQuad}

The construction of our invariants of Hitchin representations heavily relies on finite collections  of flags in $\R^n$.

\subsection{Flags}
\label{subsect:Flags}

A \emph{flag} in $\R^n$ is a family $F$ of nested linear subspaces $F^{(0)} \subset F^{(1)}\subset \dots\subset F^{(n-1)} \subset F^{(n)}$ of $\R^n$ where each $F^{(a)}$ has dimension $a$. 

A pair of flags $(E,F)$ is \emph{generic} if every subspace $E^{(a)}$ of $E$ is transverse to every subspace $F^{(b)}$ of $F$. This is equivalent to the property that $E^{(a)}\cap F^{(n-a)}=0$ for every $a$. 

Similarly, a triple of flags $(E,F,G)$ is \emph{generic} if each triple of subspaces $E^{(a)}$, $F^{(b)}$, $G^{(c)}$, respectively  in $E$, $F$, $G$, meets transversely. Again, this is equivalent to the property that $E^{(a)} \cap F^{(b)} \cap G^{(c)} =0$ for every $a$, $b$, $c$ with $a+b+c=n$. 

\subsection{Triple ratios of generic flag triples}
\label{subsect:TripleRatios}

Elementary linear algebra shows that, for any two generic flag pairs $(E,F)$ and $(E', F')$, there is a linear isomorphism $\R^n \to \R^n$ sending $E$ to $E'$ and $F$ to $F'$. However, the same is not true for generic flag triples. Indeed, there is a whole moduli space of generic flag triples modulo the action of $\PSL$, and this moduli space can be parametrized by invariants that we now describe. These invariants are expressed in terms of the exterior algebra $\Lambda (\R^n)$ of $\R^n$. 

Let $(E,F,G)$ be a generic flag triple.  For each $a$, $b$, $c$ between $0$ and $n$, the spaces $ \Lambda^a\bigl(E^{(a)}\bigr)$, $\Lambda^b\bigl(F^{(b)}\bigr)$ and $\Lambda^c\bigl(G^{(c)}\bigr)$ are all isomorphic to $\R$. Choose non-zero elements   $ e^{(a)}  \in \Lambda^a\bigl(E^{(a)}\bigr)$,  $ f^{(b)}  \in \Lambda^b\bigr(F^{(b)}\bigr)$ and  $ g^{(c)} \in \Lambda^c\bigl(G^{(c)}\bigr)$. We will use the same letters to denote their images $ e^{(a)} \in \Lambda^a(\R^n)$, $ f^{(b)} \in \Lambda^b(\R^n)$ and $ g^{(c)} \in \Lambda^c (\R^n )$.

Given integers  $a$, $b$, $c\geq 1$ with $a+b+c=n$, we then define
the \emph{$(a,b,c)$--triple ratio}  of the generic flag triple $(E,F,G)$ as the number
\begin{align*}
T_{abc} (E,F,G) &=
\frac
{ e^{(a+1)} \wedge  f^{(b)} \wedge  g^{(c-1)}}
{ e^{(a-1)} \wedge  f^{(b)} \wedge  g^{(c+1)}}
\\
&\qquad\qquad\qquad\qquad
\frac
{e^{(a)} \wedge  f^{(b-1)} \wedge  g^{(c+1)}}
{ e^{(a)} \wedge  f^{(b+1)} \wedge  g^{(c-1)}}\ 
\frac
{e^{(a-1)} \wedge  f^{(b+1)} \wedge  g^{(c)}}
{ e^{(a+1)} \wedge  f^{(b-1)} \wedge  g^{(c)}}
\end{align*}
where each of the six wedge products are elements of  $\Lambda^n(\R^n) \cong \R$. The fact that the flag triple $(E,F,G)$ is generic guarantees that these wedge products are non-zero, so that the three ratios make sense. Also, because all the spaces  $  \Lambda^{a'}\bigl(E^{(a')}\bigr)$,  $ \Lambda^{b'}\bigr(F^{(b')}\bigr)$ and  $  \Lambda^{c'}\bigl(G^{(c')}\bigr)$ involved in the expression are isomorphic to $\R$, this triple ratio is independent of the choice of the non-zero elements $ e^{(a')}  \in \Lambda^{a'}\bigl(E^{(a')}\bigr)$,  $ f^{(b')}  \in \Lambda^{b'}\bigr(F^{(b')}\bigr)$ and  $ g^{(c')} \in \Lambda^{c'}\bigl(G^{(c')}\bigr)$; indeed, each of these elements appears twice in the expression, once in a numerator and once in a denominator. 

The natural action of the linear group $\GL$ on the flag variety $\Flag$ descends to an action of the projective linear group $\PGL$, quotient of $\GL$ by its center $\R^*\Id$ consisting of all non-zero scalar multiples of the identity. Note that the projective special linear group $\PSL$ is equal to $\PGL$ if $n$ is odd, and is an index 2 subgroup of $\PGL$ otherwise.

\begin{prop}
\label{prop:TripRatiosDetermineFlagTriples}
Two generic flag triples $(E,F,G)$ and $(E', F', G')$ are equivalent under the action of $\PGL$ if and only if $ T_{abc} (E,F,G)=T_{abc} (E',F',G')$ for every $a$, $b$, $c\geq 1$ with $a+b+c=n$. 

In addition, for any set of non-zero numbers $t_{abc}\in \R^*$, there exists a generic flag triple $(E,F,G)$ such that $ T_{abc} (E,F,G)=t_{abc}$ for every $a$, $b$, $c\geq 1$ with $a+b+c=n$. 

\end{prop}

\begin{proof}
See \cite[\S9]{FoG1}.
\end{proof}

Note the elementary property of triple ratios under permutation of the flags. 

\begin{lem}
\label{lem:PermuteTripleRatios}
$$
 T_{abc} (E,F,G ) =  T_{bca} (F,G, E)= T_{bac} (F,E,G)^{-1}. 
 $$
\vskip-\belowdisplayskip
 \vskip-\baselineskip
 \qed
\end{lem}

\subsection{Quadruple ratios of generic flag triples}
\label{subsect:QuadRatios}

In addition to triple ratios, a similar type of  invariants of generic flag triples will play an important r\^ole in our analysis of Hitchin representations. 

For an integer $a$ with $1\leq a \leq n-1$, the \emph{$a$--th quadruple ratio} of the generic flag triple $(E,F,G)$ is the number
\begin{align*}
Q_a(E, F, G)
&= 
\frac
{e^{(a-1)}\wedge f^{(n-a)} \wedge g^{(1)}}
{e^{(a)}\wedge f^{(n-a-1)} \wedge g^{(1)} }\
\frac
{e^{(a)}\wedge f^{(1)} \wedge g^{(n-a-1)}}
{e^{(a-1)}\wedge f^{(1)} \wedge g^{(n-a)}}
  \\
  &\qquad\qquad\qquad\qquad \qquad \qquad\qquad
  \frac
{ e^{(a+1)}\wedge f^{(n-a-1)} }
{ e^{(a+1)}\wedge g^{(n-a-1)} }
\
\frac
{e^{(a)}\wedge g^{(n-a)} }
{e^{(a)}\wedge f^{(n-a)}  }
 \end{align*}
 where, as before, we consider arbitrary non-zero elements $ e^{(a')}  \in \Lambda^{a'}\bigl(E^{(a')}\bigr)$,  $ f^{(b')}  \in \Lambda^{b'}\bigr(F^{(b')}\bigr)$ and  $ g^{(c')} \in \Lambda^{c'}\bigl(G^{(c')}\bigr)$, and where the ratios are  computed in $\Lambda^n (\R^n) \cong \R$. As with triple ratios, the number $Q_a(E, F, G)\in \R^*$ is well-defined, independent of choices, and invariant under the action of $\PGL$ on the set of generic flag triples.

 Note that $
 Q_a(E,G,F)) = Q_a(E,F,G)^{-1}
 $, but that this quadruple ratio usually does not behave well under the other permutations of the flags $E$, $F$ and $G$, as  $E$ plays a special r\^ole in $Q_a(E,F,G)$.  
 
 By Proposition~\ref{prop:TripRatiosDetermineFlagTriples}, a generic flag triple $(E,F,G)$ is completely determined by its triple ratios modulo the action of the linear group $\GL$. It is therefore natural to expect that the quadruple ratio can be expressed in terms of the triple ratios of $(E,F,G)$. This is indeed the case, and the corresponding expression is particularly simple.

\begin{lem}
\label{lem:QuadTripleRatios}
For $a=1$, $2$, \dots, $n-1$,
$$
Q_a(E,F,G) = \prod_{b+c=n-a} T_{abc}(E,F,G)
$$
where the product is over all integers $b$, $c\geq 1$ with $b+c=n-a$. In particular, $Q_{n-1}(E,F,G) = 1$ and $Q_{n-2}(E,F,G) = T_{(n-2)11}(E,F,G)$. 
\end{lem}
\begin{proof}
When computing the right-hand side of the equation, most terms $e^{(a')} \wedge f^{(b')} \wedge g^{(c')}$  cancel out and we are left with the eight terms of $Q_a(E,F,G)$.
\end{proof}

\subsection{Double ratios of generic flag quadruples}
\label{subsect:DoubleRatios}

We now consider quadruples $(E,F,G,G')$ of flags $E$, $F$, $G$, $G'\in \Flag$.  Such a flag quadruple is \emph{generic} if each quadruple of subspaces $E^{(a)}$, $F^{(b)}$, $G^{(c)}$, $G'{}^{(d)}$ meets transversely. As usual, we can restrict attention to the cases where $a+b+c+d=n$.

For $1\leq a\leq n-1$, the \emph{$a$--th double ratio} of the generic flag quadruple $(E,F,G,G')$ is
$$
D_a (E,F,G,G') = -
\frac
{ e^{(a)} \wedge  f^{(n-a-1)}\wedge  g^{(1)}}
{ e^{(a)} \wedge  f^{(n-a-1)}\wedge  g'^{(1)}}\ 
\frac
{ e^{(a-1)} \wedge  f^{(n-a)}\wedge  g'^{(1)}}
{ e^{(a-1)} \wedge  f^{(n-a)}\wedge  g^{(1)}}
$$
where we choose arbitrary non-zero elements $ e^{(a')} \in \Lambda^{a'}(E^{(a')})$, $ f^{(b')} \in \Lambda^{b'}(F^{(b')})$, $ g^{(1)} \in \Lambda^1(G^{(1)})$ and $g'{}^{(1)}\in \Lambda^1(G'{}^{(1)})$. As usual,  $D_a (E,F,G,G') $ is independent of these choices. The minus sign is motivated by the notion of positivity that is described in \S \ref{subsect:Positivity} and plays a very important r\^ole in this article (see Proposition~\ref{prop:FlagMap}). 

\begin{lem}
\label{lem:RelationsDoubleRatios}
\begin{align*}
 D_a(E,F,G', G) &= D_a(E,F,G,G')^{-1} \\\text { and } D_a(F,E,G,G') &= D_{n-a}(E,F,G,G')^{-1}.
\end{align*}
\vskip-\belowdisplayskip
 \vskip-\baselineskip
 \qed
\end{lem}

\subsection{Positivity}
\label{subsect:Positivity}

A flag triple $(E,F,G)$ is \emph{positive} if it is generic and if all its triple ratios $T_{abc}(E,F,G)$ are positive. By Proposition~\ref{prop:TripRatiosDetermineFlagTriples}, positive flag triples form a component in the space of all generic flag triples. Lemma~\ref{lem:PermuteTripleRatios} also shows that positivity of the triple $(E,F,G)$ is preserved under all permutations of the flags $E$, $F$ and $G\in \Flag$. 

A generic flag quadruple $(E,F,G,G')$ is \emph{positive} if it is generic, if the two triples $(E,F,G)$ and $(E,F,G')$ are positive, and if all double ratios $D_a (E,F,G,G') $ are positive. 

See \cite{FoG1, Lusz2, Lusz1} for a more conceptual and general definition of positivity, valid for $k$--tuples of  flags.

\section{Invariants of Hitchin representations}
\label{sect:HitchinRepsInvariants}

We now define several invariants of Hitchin representations $\rho \colon \pi_1(S) \to \PSL$. 
These invariants require that we are given a certain topological information on the surface.

\subsection{The topological data}
\label{subsect:TopData}

Since we are going to use the terminology of geodesic laminations, it is convenient to endow the surface $S$ with a riemannian metric of negative curvature. However, it is well-known that geodesic laminations can also be defined in a metric independent way, and in particular are purely topological objects. See for instance \cite{Thu0, PenH, BonLam}.

Let $\lambda$ be a maximal geodesic lamination with finitely many leaves. Namely $\lambda$ is the union of finitely many disjoint simple closed geodesics $c_1$, $c_2$, \dots, $c_s$ and of finitely many disjoint bi-infinite geodesics $ g_1$, $ g_2$, \dots, $ g_t$ in the complement of the $c_i$, in such a way that each end of $ g_j$ spirals along some $c_i$, and that the complement $S-\lambda$ consists of finitely many infinite triangles $T_1$, $T_2$, \dots, $T_u$.

An Euler characteristic argument shows that, if $g$ is the genus of the surface $S$, then the number $u$ of  components of $S-\lambda$ is equal to $4(g-1)$, while the number $t$ of infinite leaves of $\lambda$ is equal to $6(g-1)$. The number $s$ of closed leaves of $\lambda$ can be any integer between $1$ and $3(g-1)$. 

For instance, $\lambda$ can be obtained from a family of disjoint simple closed curves $c_1$, $c_2$, \dots, $c_{3g-3}$ decomposing $S$ into pairs of pants, and then by decomposing each pair of pants into 2 infinite triangles along 3 infinite geodesics spiraling around the boundary. Figure~\ref{fig:GeodLam1} describes one such example associated with a pair of pants decomposition, and Figure~\ref{fig:GeodLam2} shows another example with only one closed leaf.

\begin{figure}[htbp]

\includegraphics{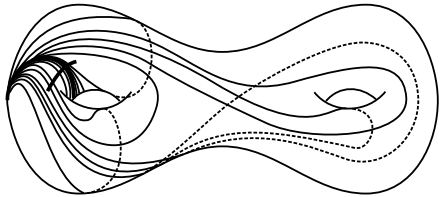}
\caption{A finite geodesic lamination with exactly one closed leaf}
\label{fig:GeodLam2}
\end{figure}

We need more data,  in addition to the finite-leaved maximal geodesic lamination $\lambda$. One is the choice of an orientation on each leaf of $\lambda$. This choice is completely free and arbitrary. In particular, we are not making any assumption of continuity  regarding these orientations, or of a relationship between these orientations and the directions in which infinite leaves spiral around closed leaves. 

Finally, for each closed leaf $c_i$, we choose an arc $k_i$ that is transverse to $\lambda$, cuts $c_i$ in exactly one point, and meets no other closed leaf $c_j$. 

For the reader who is familiar with the case where $n=2$, we can indicate that the choice of orientations for the leaves of $\lambda$ is irrelevant in that case. Regarding the need for the transverse arcs $k_i$, it is analogous to a well-known technical difficulty in the definition of the Fenchel-Nielsen coordinates: the twist parameters are relatively easy to define modulo the length parameters, but require more cumbersome topological information to be well-defined as real numbers.

\subsection{The flag curve of a Hitchin representation}
\label{subsect:FlagMap}

Once we are given this topological data, the key tool for the construction of our invariants is the Anosov structure for Hitchin representations discovered by  F. Labourie \cite{Lab1}.  Another good reference for Labourie's work is \cite{Gui}. 

We begin with what is actually a corollary of these results. 

\begin{prop}[Labourie {\cite{Lab1}}]
\label{prop:HitchinLoxodromic}
Let $\rho \colon \pi_1(S) \to \PSL$ be a Hitchin representation. Then, for every non-trivial $\gamma \in \pi_1(S)$, the element $\rho(\gamma)\in \PSL$ has real eigenvalues and their absolute values are distinct. \qed
\end{prop}

When $n$ is even, the eigenvalues of $\rho(\gamma) \in \PSL=\SL/\{\pm\Id\}$ are only defined up to sign. However, we can be a little more specific. 

\begin{lem}
\label{lem:HitchinLoxodromicPositiveEigenvalues}
Let $\rho \colon \pi_1(S) \to \PSL$ be a Hitchin representation. Then, for every non-trivial $\gamma \in \pi_1(S)$, the element $\rho(\gamma)\in \PSL$ admits a lift $\rho(\gamma)'\in \SL$ whose eigenvalues are distinct and all positive. 
\end{lem}

\begin{proof}
Let the non-trivial element $\gamma\in \pi_1(S)$ be fixed. 

The property ``all eigenvalues of a lift $\rho(\gamma)'\in \SL$ of $\rho(\gamma)\in \PSL$ have the same sign'' is open and closed in the space of Hitchin representations, since these eigenvalues are real and non-zero. This property holds in the special case where $\rho$ is the composition of a Teichm\"uller representation $\rho_2 \colon \pi_1(S) \to \mathrm{PSL}_2(\R)$ with the natural embedding $\mathrm{PSL}_2(\R) \to \PSL$; indeed, if $\rho_2(\gamma)$ has a lift $\rho_2(\gamma)'\in \mathrm{SL}_2(\R)$ with eigenvalues $a$ and $a^{-1}$, then $\rho(\gamma)$ has a lift $\rho(\gamma)'\in \SL$ with eigenvalues $a^{n-2k+1}$ as $k$ ranges over all integers with $1\leq k \leq n$. Therefore, the property  holds for every Hitchin representation by connectedness of the  Hitchin component.

This proves that  the eigenvalues of any lift  $\rho(\gamma)'\in \SL$ of $\rho(\gamma)\in \PSL$ have the same sign. If these eigenvalues are all negative, note that   $n$ is even since $\rho(\gamma)'$ has determinant $+1$. Then $-\rho(\gamma)'\in \SL$ is another lift of $\rho(\gamma)\in \PSL$, whose eigenvalues are all positive (and distinct by Proposition~\ref{prop:HitchinLoxodromic}).  
\end{proof}

If $\rho$ is a Hitchin representation and if $\gamma \in \pi_1(S)$ is non-trivial, let $\rho(\gamma)'\in \SL$ be the lift of $\rho(\gamma) \in \PSL$ given by Lemma~\ref{lem:HitchinLoxodromicPositiveEigenvalues}. Let
$$
m_1^\rho(\gamma) > m_2^\rho(\gamma) > \dots > m_n^\rho(\gamma) >0
$$
be the eigenvalues of $\rho(\gamma)'$, indexed in decreasing order. Since these eigenvalues are distinct, $\rho(\gamma)'$ is diagonalizable. Let $L_a$ be the (1--dimensional) eigenspace corresponding to the eigenvalue $m_a^\rho(\gamma) $. 

This associates to $\rho(\gamma)$ two preferred flags $E$, $F\in \Flag$ defined by the property that
$$
E^{(a)} = \bigoplus_{b=1}^a L_b \ \text{ and }\  F^{(a)} = \kern -5pt \bigoplus_{b=n-a+1}^n \kern -5pt L_b.
$$
By definition, $E$ is the \emph{stable flag} of $\rho(\gamma)\in \PSL$, and $F$ is its \emph{unstable flag}. 

Let $\widetilde S$ be the universal covering of the surface $S$, and let $\partial_\infty \widetilde S$ be its circle at infinity. Recall that every non-trivial $\gamma \in \pi_1(S)$ fixes two points of $\partial_\infty \widetilde S$, one of them attracting and the other one repelling. 

\begin{prop}[Labourie, Fock-Goncharov]
\label{prop:FlagMap}

Given a Hitchin representation $\rho \colon \pi_1(S) \to \PSL$, there exists 
 a unique continuous map $\mathcal F_\rho \colon \partial_\infty \widetilde S \to \Flag$ such that:

 \begin{enumerate}
 
 \item if $x\in \partial_\infty \widetilde S$ is the attracting fixed point of $\gamma\in \pi_1(S)$, then $\mathcal F_\rho(x) \in \Flag$ is the stable flag of $\rho(\gamma) \in \PSL$; 
 
 \item $\mathcal F_\rho$ is equivariant with respect to the Hitchin homomorphism $\rho\colon \pi_1(S) \to \PSL$, in the sense that $\mathcal F_\rho(\gamma x)= \rho(\gamma)(x)$ for every $\gamma \in \pi_1(S)$ and every $x \in \partial _\infty \widetilde S$;
 
\item for any two distinct points $x$, $y \in  \partial_\infty \widetilde S$, the flag pair $\bigl( \mathcal F_\rho(x),  \mathcal F_\rho(y) \bigr)$ is generic;

\item for any three distinct points $x$, $y$, $z \in  \partial_\infty \widetilde S$, the flag triple $\bigl( \mathcal F_\rho(x), \mathcal F_\rho(y) , \mathcal F_\rho(z)\bigr)$ is positive;

\item for any four distinct points $x$, $y$, $z $, $z'$ occurring in this order around the circle at infinity $  \partial_\infty \widetilde S$, the flag quadruple $\bigl( \mathcal F_\rho(x), \mathcal F_\rho(y) , \mathcal F_\rho(z), \mathcal F_\rho(z') \bigr)$ is positive. \qed

\end{enumerate}

\end{prop}

By definition, this curve  $\mathcal F_\rho \colon \partial_\infty \widetilde S \to \Flag$ is the \emph{flag curve} of the Hitchin representation $\rho \colon \pi_1(S) \to \PSL$. 

The first three properties  of Proposition~\ref{prop:FlagMap} are immediate consequences of the Anosov structure of \cite{Lab1}. The  positivity properties of the last two conditions of Proposition~\ref{prop:FlagMap} were proved by Fock and Goncharov \cite{FoG1}; see also the hyperconvexity property of \cite{Lab1, Gui}. 

\subsection{Invariants of triangles}
\label{subsect:TriangleInvariants}

Given a finite maximal geodesic lamination $\lambda$ as in \S \ref{subsect:TopData} and a Hitchin representation $\rho \colon \pi_1(S) \to \PSL$, the  first set of invariants of $\rho$ is associated with the components of the complement $S-\lambda$. Recall that each of these components is an ideal triangle.

Consider such a triangle  $T_j$, and select one of its vertices $v_j$.  (Such a vertex is of course not an actual point of the surface $S$; we let the reader devise a formal definition for a vertex of the ideal triangle $T_j \subset S$.) Lift $T_j$ to  an ideal triangle $\widetilde T_j$ in the universal covering $\widetilde S$, and let $\widetilde v_j \in \partial_\infty \widetilde S$ be the vertex of $\widetilde T_j$ corresponding to the vertex $v_j$ of $T_j$.  Label the vertices of $\widetilde T_j$ as $\widetilde v_j$, $\widetilde v_j'$ and $\widetilde v_j'' \in \partial_\infty \widetilde S$ in clockwise order around $\widetilde T_j$. Then, by Proposition~\ref{prop:FlagMap},  the  flag triple $\bigl( \mathcal F_\rho(\widetilde v_j), \mathcal F_\rho(\widetilde v_j') , \mathcal F_\rho(\widetilde v_j'')\bigr)$ is positive, and we can  consider the  logarithms
$$
\tau_{abc}^\rho (T_j, v_j) = \log T_{abc}
\bigl( \mathcal F_\rho(\widetilde v_j), \mathcal F_\rho(\widetilde v_j') , \mathcal F_\rho(\widetilde v_j'')\bigr)
$$
to the triple ratios of this flag triple, defined for  every $a$, $b$, $c\geq 1$ with $a+b+c=n$. By $\rho$--equivariance of the flag curve $\mathcal F_\rho$, these triple ratio   logarithms depend only on the triangle $T_j$ and on the vertex $v_j$ of $T_j$, and not on  the choice of the lift $\widetilde T_j$.

Lemma~\ref{lem:PermuteTripleRatios} indicates how the invariant $\tau_{abc}^\rho (T_i, v_j)  \in \R$ changes if we choose a different vertex of the triangle $T_j$. 

\begin{lem}
\label{lem:RotateTirangleInvariants}
If $v_j$, $v_j'$ and $v_j''$ are the vertices of $T_j$, indexed clockwise around $T_j$, then
$$
\tau_{abc}^\rho (T_i, v_j)  = \tau_{bca}^\rho (T_i, v_j')  =\tau_{cab}^\rho (T_i, v_j'') .
$$
\vskip-\belowdisplayskip
\vskip - \baselineskip
\qed
\end{lem}

\subsection{Shear invariants of infinite leaves}
\label{subsect:InfiniteLeafInvariants}

Let $ g_j$ be an infinite leaf of $\lambda$. 

Lift $ g_j$ to a leaf $\widetilde  g_j$ of the preimage $\widetilde\lambda$ of $\lambda$ in the universal covering $\widetilde S$. This leaf is isolated in $\widetilde\lambda$, and is adjacent to two components $\widetilde T$ and $\widetilde T'$ of the complement $\widetilde S - \widetilde\lambda$. Choose the notation so that $\widetilde T$ and $\widetilde T'$ are respectively to the left and to the right of $\widetilde g_j$ for the orientation of $\widetilde g_j$ coming from the orientation  of $ g_j$ that is part of the topological data of \S \ref{subsect:TopData}. 

Let $x$ and $y\in \partial_\infty \widetilde S$ be the positive and negative end points of $\widetilde  g_j$. Let $z$, $z'\in \partial_\infty \widetilde S$ be the third vertices of $\widetilde T$ and $\widetilde T'$, respectively, namely the vertices of these triangles that are neither $x$ nor $y$. See Figure~\ref{fig:ShearInfinite}. Consider the flags $E=\mathcal F_\rho(x)$, $F=\mathcal F_\rho(y)$, $G=\mathcal F_\rho(z)$ and $G'=\mathcal F_\rho(z')$ associated to these vertices by the flag curve $\mathcal F_\rho \colon \partial_\infty \widetilde S \to \Flag$. 

For $1\leq a \leq n-1$, we can now consider the double ratio $D_a(E, F, G, G')$ as in \S \ref{subsect:DoubleRatios}. This double ratio is positive by Proposition~\ref{prop:FlagMap}. The  \emph{$a$--th shear invariant} of the Hitchin homomorphism $\rho$ along the oriented leaf $ g_j$ is then defined as
the logarithm
$$
\sigma_a^\rho( g_j) = \log D_a(E, F, G, G').
$$
This invariant $\sigma_a( g_j) $ is clearly independent of the choice of the lift $\widetilde  g_j$ of the leaf $ g_j$ to $\widetilde S$, by $\rho$--equivariance of the flag curve $\mathcal F_\rho$. 

By Lemma~\ref{lem:RelationsDoubleRatios},  reversing the orientation of $ g_j$ replaces $\sigma_a^\rho( g_j) $ by $\sigma_{n-a}^\rho( g_j) $.

\begin{figure}[htbp]

\SetLabels
( .6*.58 ) $\widetilde g_j $ \\
(.53 *.67 ) $\widetilde T $ \\
( .63* .37) $\widetilde T'$ \\
( 1*.63 ) $x $ \\
( 0* .46) $ y$ \\
( .42*1 ) $ z$ \\
\T( .63*0.02 ) $z' $ \\
\endSetLabels
\centerline{\AffixLabels{ \includegraphics{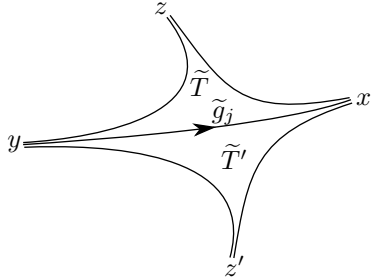} }}

\caption{The construction of shear invariants of infinite leaves}
\label{fig:ShearInfinite}
\end{figure}

\subsection{Shear invariants of closed leaves}
\label{subsect:ClosedLeafInvariant}

The shear invariants of a closed leaf $c_i$ are defined in very much the same way as for infinite leaves, except that we need to use the transverse arc $k_i$ that is part of the topological data to single out two triangles $T$ and $T'$ that are located on either side of $c_i$.

More precisely, let $\widetilde c_i$ be a component of the pre-image of $c_i$ in the universal cover $\widetilde S$, and orient it by the orientation of $c_i$. Lift the arc $k_i$ to an arc $\widetilde k_i$ that meets $\widetilde c_i$ in one point. Let $\widetilde T$ and $\widetilde T'$ be the two triangle components of $\widetilde S - \widetilde \lambda$ that contain the end points of $\widetilde k_i$, in such a way that $\widetilde T$ and $\widetilde T'$ are respectively to the left and to the right of $\widetilde c_i$ for the orientation of $\widetilde c_i$ lifting the orientation of $c_i$. 

Let $x$ and $y\in \partial_\infty \widetilde S$ be the positive and negative end points of $\widetilde  c_i$, respectively. Among the vertices of $\widetilde T$, let $z \in \partial_\infty \widetilde S$ be the one that is farthest away from $\widetilde c_i$; namely, $z$ is adjacent to the two components of $\widetilde S - \widetilde T$ that do not contain $\widetilde c_i$. Similarly, let  $z'$ be the vertex of $T'$  that is farthest away from $\widetilde c_i$. See Figures~\ref{fig:ShearClosed}(a) and (b) for two of the four possible configurations, according to the directions of the spiraling of infinite leaves around $ c_i$. 

\begin{figure}[htbp]

\SetLabels
( .235*.35 ) $\widetilde k_i $ \\
(.2 *.69 ) $\widetilde T $ \\
( .26* .25) $\widetilde T'$ \\
( .46*.5 ) $x $ \\
( 0* .5) $ y$ \\
( .16*1.01 ) $ z$ \\
\T( .26*0.02 ) $z' $ \\
( .775*.35 ) $\widetilde k_i $ \\
(.735 *.69 ) $\widetilde T $ \\
( .75* .25) $\widetilde T'$ \\
( 1.005*.5 ) $x $ \\
( .53* .5) $ y$ \\
( .695*1.01 ) $ z$ \\
\T( .74*0.02 ) $z' $ \\
( .22*-.2 ) (a) \\
( .8*-.2 ) (b) \\
\endSetLabels
\vskip 5pt
\centerline{\AffixLabels{ \includegraphics{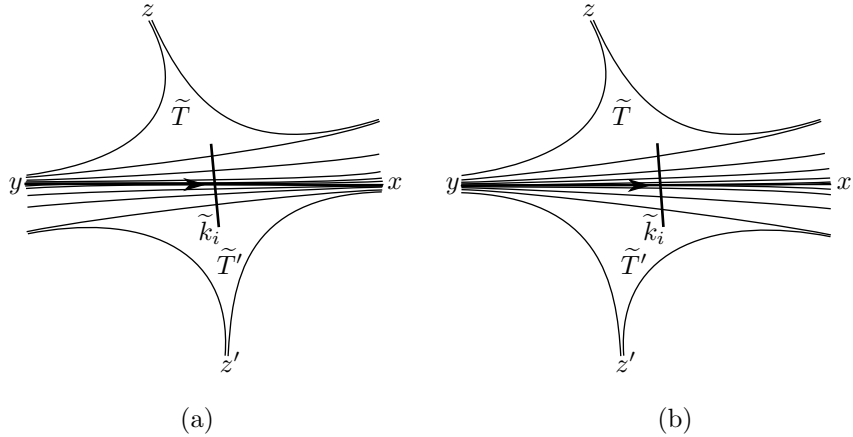} }}
\vskip 20pt

\caption{The construction of shear invariants of closed leaves}
\label{fig:ShearClosed}
\end{figure}

Finally, let $E=\mathcal F_\rho(x)$, $F=\mathcal F_\rho(y)$, $G=\mathcal F_\rho(z)$ and $G'=\mathcal F_\rho(z')$ be the flags associated to these vertices by the flag curve $\mathcal F_\rho \colon \partial_\infty \widetilde S \to \Flag$. 

With this data, we again consider for $1\leq a \leq n-1$  the double ratio $D_a(E, F, G, G')>0$ as in \S \ref{subsect:DoubleRatios} and  Proposition~\ref{prop:FlagMap}. The  \emph{$a$--th shear invariant} of the Hitchin homomorphism $\rho$ along the oriented closed leaf $ c_i$ is defined as
the logarithm
$$
\sigma_a^\rho( c_i) = \log D_a(E, F, G, G').
$$
By $\rho$--equivariance of the flag curve $\mathcal F_\rho$, this shear invariant $\sigma_a( c_i) $ is clearly independent of the choice of the component $\widetilde c_i$ of the preimage of $ c_i$, and of the lift $\widetilde k_i$ of the arc $k_i$. 

Again, Lemma~\ref{lem:RelationsDoubleRatios} shows that  reversing the orientation of $ c_i$ replaces $\sigma_a^\rho( c_i) $ by $\sigma_{n-a}^\rho( c_i) $. 

\subsection{Lengths of closed leaves}
\label{subsect:Lengths}

There are simpler invariants that we could have considered. 

Let $ c$ be an oriented closed curve in $S$ that is not homotopic to 0, and let $[c]\in \pi_1(S)$ be represented by $c$ after connecting this curve to the base point by an arbitrary path. 

If $\rho$ is a Hitchin representation, Lemma~\ref{lem:HitchinLoxodromicPositiveEigenvalues} asserts that $\rho\bigl([c]\bigr)\in \PSL$ admits a lift $\rho\bigl([c]\bigr)'\in \SL$ with distinct real eigenvalues 
$$
m_1^\rho(c) > m_2^\rho(c) > \dots > m_n^\rho(c) >0.
$$

For $1\leq a \leq n-1$, we define the \emph{$a$--th $\rho$--length} $\ell_a^\rho( c)$ of the closed curve $ c$ for the Hitchin representation $\rho \colon \pi_1(S) \to \PSL$ as
$$
\ell_a^\rho( c) =  \log \frac{m_a^\rho( c) }{m_{a+1}^\rho( c)}>0.
$$

Note that this quantity is independent of the class $[c]\in \pi_1(S)$ represented by $c$, and that reversing the orientation of $ c$ replaces $\ell_a^\rho( c) $ by $\ell_{n-a}^\rho( c) $. 

The $\rho$--lengths  $\ell_a^\rho( c_i) $ of the closed leaves $c_i$ of the geodesic lamination $\lambda$ will play an important r\^ole in the next sections. 

\begin{rem}
The reader should beware of a discrepancy between the conventions of this paper and those of  \cite{Dre}: What we call here $\ell_a^\rho( c) $ is called $\ell_a^\rho( c) -\ell_{a+1}^\rho( c) $ in \cite{Dre}. 
There are two reasons for this change in conventions. The main one is that, as we will see in \S \ref{sect:RelationsInvariants}, the $\rho$--lengths $\ell_a^\rho( c_i) $ of the closed leaves of $\lambda$ are related to the triangle and shear invariants of $\rho$, and the expression of this connection is simpler with the current definitions. The second reason comes from the case $n=2$, where the representation $\rho$ defines a hyperbolic metric $m$ on the surface $S$; then there is exactly one  length function and this $\rho$--length $\ell_1^\rho( c) $ is exactly the classical length of the $m$--geodesic that is homotopic to $ c$, which plays a fundamental r\^ole in much of hyperbolic geometry. 

\end{rem}

\section{Relations between invariants}
\label{sect:RelationsInvariants}

Let $ c_i$ be a closed leaf of the geodesic lamination $\lambda$. We will express the $\rho$--lengths $\ell_a( c_i)$ in terms of the triangle invariants $\tau_{abc}^\rho (T_j, v_j)$ and of the shear invariants $\sigma_{n-a}^\rho( g_j) $ of the infinite leaves $ g_j$. 

The closed leaf $ c_i$ has two sides  $ c_i^\Right$ and  $ c_i^\Left$, respectively located to the  right and to the left of $ c_i$ for the chosen orientation of this curve. 

Select one of these sides $ c_i^\Right$ or  $ c_i^\Left$. 
Let   $ g_{i_1}$, $ g_{i_2}$, \dots, $ g_{i_k}$ be the infinite leaves of $\lambda$ that spiral on this side of $ c_i$. An infinite leaf $ g_j$ will appear twice in this list if its two ends spiral on the selected side of $ c_i$. 
We can then consider the shear invariants $\sigma_a^\rho( g_{i_l})\in \R$. 

 Similarly, let $T_{j_1}$, $T_{j_2}$, \dots, $T_{j_k}$ be the components of the complement $S-\lambda$ that spiral on the selected  side $ c_i^\Right$ or  $ c_i^\Left$  of $ c_i$. The spiraling of $T_{j_l}$ around this side occurs in the direction of a vertex $v_l$ of   $T_{j_l}$. We can then consider the triangle invariants $\tau_{abc}(T_{j_l}, v_l)$, as in \S \ref{subsect:TriangleInvariants}.

 We have to worry about orientations, and more precisely about two types of orientation. One is the orientation of each spiraling leaf $ g_{i_l}$. The other is whether the spiraling occurs in the direction of the orientation of $ c_i$ or not.

\begin{prop}
\label{prop:LengthFunctionOtherInvariants}
Select a side $ c_i^\Right$ or  $ c_i^\Left$  of the closed leaf $ c_i$ of $\lambda$. 
Let  $ g_{i_l}$ and $T_{j_l}$, $l=1$, $2$, \dots, $k$ be the infinite leaves and triangles that spiral on this side of $c_i$. Let $v_l$ be the vertex of the  triangle $T_{j_l}$  in the direction of which the spiraling occurs, and consider the triangle invariants $\tau_a^\rho(T_{j_l}, v_l)$ and the shear invariants $ \sigma_a^\rho( g_{i_l})$. Set 
$\overline\sigma_a^\rho( g_{i_l}) = \sigma_a^\rho( g_{i_l})$ if the leaf $ g_{i_l}$ is oriented towards $ c_i$, and $\overline\sigma_a^\rho( g_{i_l}) = \sigma_{n-a}^\rho( g_{i_l})$ if it is oriented away from $ c_i$. Then:

 \begin{enumerate}
\item  if the selected side is the right-hand side $c_i^\Right$  and if the spiraling occurs in the direction of the orientation of~$ c_i$, 
$$\ell_a^\rho( c_i) = \sum_{l=1}^k \overline\sigma_{ a}^\rho( g_{i_l}) +   \sum_{l=1}^k \sum_{b+c=n-a} \tau_{abc}^\rho (T_{j_l}, v_l);$$

\item  if the selected side is the right-hand side $c_i^\Right$  and if the spiraling occurs in the direction opposite to the orientation of~$ c_i$,
$$
 \ell_a^\rho( c_i) 
 =- \sum_{l=1}^k
\overline
\sigma_{n-a}^\rho ( g_{i_l})
- \sum_{l=1}^k \sum_{b+c=a} \tau_{(n-a)bc}^\rho (T_{j_l}, v_l);
$$

\item  if the selected side is the left-hand side $c_i^\Left$ and if the spiraling occurs in the direction of the orientation of~$ c_i$,
$$\ell_a^\rho( c_i) = -\sum_{l=1}^k \overline\sigma_{ a}^\rho( g_{i_l}) -   \sum_{l=1}^k \sum_{b+c=n-a} \tau_{abc}^\rho (T_{j_l}, v_l);$$

\item  if the selected side is the left-hand side $c_i^\Left$  and if the spiraling occurs in the direction opposite to the orientation of $ c_i$,
$$
 \ell_a^\rho( c_i) 
 = \sum_{l=1}^k
\overline
\sigma_{n-a}^\rho ( g_{i_l})
+ \sum_{l=1}^k \sum_{b+c=a} \tau_{(n-a)bc}^\rho (T_{j_l}, v_l).
$$

\end{enumerate}

\end{prop}

\begin{proof} First consider the case where the selected side is the right-hand side $c_i^\Right$,  and where  the spiraling occurs in the direction of the orientation of $ c_i$.

Without loss of generality, we can assume that the infinite leaves $ g_{i_1}$, $ g_{i_2}$, \dots, $ g_{i_k}$, $ g_{i_{k+1}}= g_{i_1}$ and triangles $T_{j_1}$, $T_{j_2}$, \dots, $T_{j_k}$, $T_{j_{k+1}}=T_{j_1}$ are indexed so that the spiraling part of $T_{j_l}$ is bounded on the left by $ g_{i_{l-1}}$ and on the right by $ g_{i_{l}}$.

In the universal covering $\widetilde S$, lift  the leaves $  g_{i_l}$ to leaves $\widetilde  g_{i_l}$ of the preimage $\widetilde \lambda$ of $\lambda$,  and the triangles $T_{j_l}$ to triangles $\widetilde T_{j_l}$, in such a way that $\widetilde g_{i_{l-1}}$  and $\widetilde g_{i_l}$ are two of the sides of $\widetilde T_{j_l}$. See Figure~\ref{fig:Spiraling}.  Then, because of the orientation choice for $ c_i$, we have that $ [c_i](\widetilde  g_{i_1}) = \widetilde  g_{i_{k+1}}$ and  $[ c_i](\widetilde T_{j_1}) = \widetilde T_{j_{k+1}}$ for a suitable class $[ c_i] \in \pi_1(S)$ represented by the curve $ c_i$. 

\begin{figure}[htbp]

\SetLabels
( .5*.85 ) $\widetilde c_i $ \\
( .19*.46) $\widetilde T_{j_1} $ \\
( .35*.52) $\widetilde g_{i_1} $ \\
(.25 * .28) $\widetilde T_{j_2} $ \\
( .38*.36 ) $\widetilde g_{i_2} $ \\
( .775* .36) $\widetilde g_{i_k} $ \\
( .715* .18) $\widetilde T_{j_k} $ \\
( .67* .38) $\widetilde g_{i_{k-1}} $ \\
( .835* .19) $\widetilde T_{j_{k+1}} $ \\
( .86* .34) $\widetilde g_{i_{k+1}} $ \\
( 1* .82) $ x$ \\
(.06 * .3) $x_1 $ \\
(.18 * .06) $x_2 $ \\
\T(.62 * 0) $ x_{k-1}$ \\
\T( .745* 0) $ x_k$ \\
\T(.89 * 0) $ x_{k+1}$ \\
( * ) $ $ \\
( * ) $ $ \\
( * ) $ $ \\
( * ) $ $ \\
( * ) $ $ \\
( * ) $ $ \\
( * ) $ $ \\
\endSetLabels
\centerline{\AffixLabels{ \includegraphics{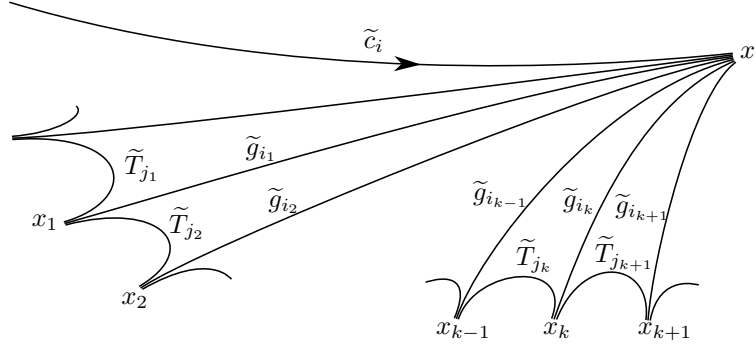} }}

\caption{Spiraling in the universal cover $\widetilde S$}
\label{fig:Spiraling}
\end{figure}

Let $E = \mathcal F_\rho(x) \in \Flag$ be the flag associated by the flag curve $\mathcal F_\rho$  to the common end point $x \in \partial_\infty \widetilde S$ of the $\widetilde g_{i_l}$. Similarly, let $F_l = \mathcal F_\rho(x_l)$ be the flag associated with the other end point $x_l \not = x$ of $\widetilde g_{i_l}$. It is also convenient to set $G_l = F_{l-1} \in \Flag$.  

Pick non-zero elements $e^{(a')} \in \Lambda^{a'} \bigl( E^{(a')} \bigr)$, $f_{l'}^{(a')} \in \Lambda^{a'} \bigl( F_{l'}^{(a')} \bigr)$ and $g_{l'}^{(a')} \in \Lambda^{a'} \bigl( G_{l'}^{(a')} \bigr)  =  \Lambda^{a'} \bigl( F_{l'-1}^{(a')} \bigr)$. 

\begin{lem}
\label{lem:RewriteShear}
\begin{align*}
\overline\sigma_a^\rho ( g_{i_l}) &= \log 
\Biggl|
\frac
{e^{(a)}\wedge f_l^{(n-a-1)} \wedge g_l^{(1)} }
 {e^{(a-1)}\wedge f_l^{(n-a)} \wedge g_l^{(1)}}
 \ 
 \frac
 {e^{(a-1)}\wedge f_{l+1}^{(1)} \wedge g_{l+1}^{(n-a)}}
{e^{(a)}\wedge f_{l+1}^{(1)} \wedge g_{l+1}^{(n-a-1)} }\\
&\qquad\qquad\qquad\qquad\qquad\qquad\qquad\qquad
 \frac
 {e^{(a)}\wedge f_{l}^{(n-a)} }
  {e^{(a)}\wedge g_{l+1}^{(n-a)} }\ 
\frac{ e^{(a+1)}\wedge g_{l+1}^{(n-a-1)} }
{e^{(a+1)}\wedge f_{l}^{(n-a-1)} }
\Biggr|
\end{align*}

\end{lem}
\begin{proof}
The right-hand side of the equation is clearly independent of the choices of elements $e^{(a')} \in \Lambda^{a'} \bigl( E^{(a')} \bigr)$, $f_{l'}^{(a')} \in \Lambda^{a'} \bigl( F_{l'}^{(a')} \bigr)$ and $g_{l'}^{(a')} \in \Lambda^{a'} \bigl( F_{l'-1}^{(a')} \bigr)$. It is equal to the logarithm of the double ratio $ D_a(E, F_l, F_{l-1}, F_{l+1})$ of \S \ref{subsect:DoubleRatios} in the special case where $g_{l'}^{(a')} = f_{l'-1}^{(a')}$, and therefore in all cases. 

(The absolute value is introduced so that, here and later in the arguments,  we do not have to worry about sign changes when we split ratios or permute terms in the wedge products.)

By definition of $\sigma_a^\rho ( g_{i_l})$, this logarithm $\log D_a(E, F_l, F_{l-1}, F_{l+1})$  is equal to $\sigma_a^\rho ( g_{i_l})$ when  the orientation of $\widetilde g_{i_l}$ points towards $x \in \partial_\infty \widetilde S$, and is equal to  $\sigma_{n-a}^\rho( g_{i_l})$ when $ g_{i_l}$ is oriented away from $x$. It is therefore equal to  $\overline \sigma_a^\rho( g_{i_l})$ in all cases. 
\end{proof}

Lemma~\ref{lem:RewriteShear} enables us to split $\overline\sigma_a^\rho ( g_{i_l})$ according to the respective contributions of the triangles $\widetilde T_{i_l}$ and $\widetilde T_{i_{l+1}}$:
\begin{align*}
\overline\sigma_a^\rho ( g_{i_l}) &= \log \,
\Biggl|
\frac
{e^{(a)}\wedge f_l^{(n-a-1)} \wedge g_l^{(1)} }
{e^{(a-1)}\wedge f_l^{(n-a)} \wedge g_l^{(1)}}
\frac
{e^{(a)}\wedge f_{l}^{(n-a)} }
{e^{(a+1)}\wedge f_{l}^{(n-a-1)} }
 \Biggr|
 \\
 &\qquad\qquad\qquad\qquad
 + \log \,
\Biggl|
\frac
{e^{(a-1)}\wedge f_{l+1}^{(1)} \wedge g_{l+1}^{(n-a)}}
{e^{(a)}\wedge f_{l+1}^{(1)} \wedge g_{l+1}^{(n-a-1)}}
  \frac
  {e^{(a+1)}\wedge g_{l+1}^{(n-a-1)} }
{e^{(a)}\wedge g_{l+1}^{(n-a)} }
 \Biggr|
\end{align*}

Summing over $l$ and grouping terms according to the contribution of each triangle gives
\begin{align*}
\sum_{l=1}^k
\overline\sigma_a^\rho ( g_{i_l})&= \log \,
\Biggl|
\frac
{e^{(a)}\wedge f_1^{(n-a-1)} \wedge g_1^{(1)} }
  {e^{(a-1)}\wedge f_1^{(n-a)} \wedge g_1^{(1)}}
  \,
  \frac
  {e^{(a)}\wedge f_{1}^{(n-a)} }
{e^{(a+1)}\wedge f_{1}^{(n-a-1)}  }
 \Biggr|
\\
 &\qquad+
 \sum_{l=2}^k
  \log \,
\Biggl|
\frac
{e^{(a)}\wedge f_l^{(n-a-1)} \wedge g_l^{(1)}}
  {e^{(a-1)}\wedge f_l^{(n-a)} \wedge g_l^{(1)}}
  \,
  \frac
  {e^{(a)}\wedge f_{l}^{(n-a)} }
{e^{(a+1)}\wedge f_{l}^{(n-a-1)} }
  \\
  &\qquad\qquad\qquad\qquad\qquad\qquad
\frac
{e^{(a-1)}\wedge f_{l}^{(1)} \wedge g_{l}^{(n-a)}}
{e^{(a)}\wedge f_{l}^{(1)} \wedge g_{l}^{(n-a-1)}}
  \,
  \frac
  { e^{(a+1)}\wedge g_{l}^{(n-a-1)} }
{e^{(a)}\wedge g_{l}^{(n-a)} }
   \Biggr|
  \\
 &\qquad+ \log \,
\Biggl|
\frac
{e^{(a-1)}\wedge f_{k+1}^{(1)} \wedge g_{k+1}^{(n-a)}}
  {e^{(a)}\wedge g_{k+1}^{(n-a)} }
  \,
  \frac
  { e^{(a+1)}\wedge g_{k+1}^{(n-a-1)} }
{e^{(a)}\wedge f_{k+1}^{(1)} \wedge g_{k+1}^{(n-a-1)}}
 \Biggr|
\end{align*}

Consider the last term. Lift $\rho\bigl( [c_i] \bigr)\in \PSL$ to $\rho\bigl( [c_i] \bigr)'\in \SL$. Remembering that $\widetilde T_{j_{k+1}}$ is equal to $ c_i(\widetilde T_{j_1})$, we can choose $f_{k+1}^{(1)}= \rho\bigl( [c_i] \bigr)_*' \bigl (f_{1}^{(1)}\bigr)$ and $g_{k+1}^{(a')}= \rho\bigl( [c_i] \bigr)_*' \bigl(g_{1}^{(a')}\bigr)$ for every $a'$, where $\rho\bigl( [c_i] \bigr)_*'\colon \Lambda^{a'}(\R^n) \to  \Lambda^{a'}(\R^n)$ is the map induced by $\rho\bigl( [c_i] \bigr)'\colon \R^n \to \R^n$. Also, $\rho\bigl( [c_i] \bigr)'$ respects the flag $E$. 

We now use the fact that  the spiraling occurs in the direction of the orientation of $ c_i$.   As in \S \ref{subsect:FlagMap}, $\rho\bigl( [c_i] \bigr)'$  has  eigenvalues $m_1^\rho( c_i)$, $m_2^\rho( c_i)$, \dots, $m_n^\rho( c_i)$ with corresponding 1-dimensional  eigenspaces $L_1$, $L_2$, \dots, $L_n$. Because the flag $E=\mathcal F_\rho(x)$ is associated with the  positive end point $x$ of a component $\widetilde c_i$ of the preimage of $c_i$, it is the stable flag of $\rho\bigl( [c_i] \bigr)$ by Part~(1) of Proposition~\ref{prop:FlagMap} and 
$$
E^{(a)} = \bigoplus_{b=1}^a L_b. 
$$
Therefore
$$
 \rho\bigl( [c_i] \bigr)'_*\bigl(e^{(a)}\bigr)=\biggl( \prod_{b=1}^{a} m_b^\rho( c_i)\biggr)   e^{(a)}
$$
for every  $ e^{(a)} \in \Lambda^{a}(E^{(a)})$.

Then, using the fact that $\rho\bigl( [c_i] \bigr)'\in \SL$ acts by the identity on $\Lambda^n(\R^n)$ for the second equation,
\begin{align*}
 \log \,&
\Biggl|
\frac
{e^{(a-1)}\wedge f_{k+1}^{(1)} \wedge g_{k+1}^{(n-a)}}
  {e^{(a)}\wedge f_{k+1}^{(1)} \wedge g_{k+1}^{(n-a-1)}}
  \,
  \frac
  {e^{(a+1)}\wedge g_{k+1}^{(n-a-1)} }
{e^{(a)}\wedge g_{k+1}^{(n-a)} }
 \Biggr|
 \\
 &=
 \log \,
\Biggl|
\frac
{e^{(a-1)}\wedge  \rho\bigl( [c_i] \bigr)_*'\bigl(f_{1}^{(1)} \bigr) \wedge  \rho\bigl( [c_i] \bigr)'_*\bigl(g_{1}^{(n-a)} \bigr) 
  }
  {e^{(a)}\wedge \rho\bigl( [c_i] \bigr)'_*\bigl( f_{1}^{(1)} \bigr) \wedge  \rho\bigl( [c_i] \bigr)'_* \bigl(g_{1}^{(n-a-1)} \bigr)}
  \,
  \frac
  { e^{(a+1)}\wedge  \rho\bigl( [c_i] \bigr)'_* \bigl(g_{1}^{(n-a-1)} \bigr)}
{e^{(a)}\wedge \rho\bigl( [c_i] \bigr)'_* \bigl( g_{1}^{(n-a)} \bigr) }
 \Biggr|
 \\
 &=
 \log \,
\Biggl|
\frac
{ \rho\bigl( [c_i] \bigr)'_*{}^{-1}\bigl(e^{(a-1)} \bigr)\wedge f_{1}^{(1)} \wedge g_{1}^{(n-a)}
  }
  { \rho\bigl( [c_i] \bigr)'_*{}^{-1}\bigl(e^{(a)} \bigr) \wedge  f_{1}^{(1)} \wedge g_{1}^{(n-a-1)}}
  \,
  \frac
  { \rho\bigl( [c_i] \bigr)'_*{}^{-1} \bigl(e^{(a+1)} \bigr) \wedge  g_{1}^{(n-a-1)} }
{\rho\bigl( [c_i] \bigr)'_*{}^{-1} \bigl(e^{(a)} \bigr) \wedge g_{1}^{(n-a)}   }
 \Biggr|
 \\
 &=
 \log \,
\Biggl|
\frac{m_a^\rho( c_i)}{m_{a+1}^\rho( c_i)}\ 
\frac
{e^{(a-1)}\wedge f_{1}^{(1)} \wedge g_{1}^{(n-a)}}
  {  e^{(a)} \wedge  f_{1}^{(1)} \wedge g_{1}^{(n-a-1)}}
  \
\frac  { e^{(a+1)} \wedge  g_{1}^{(n-a-1)} }
{ e^{(a)} \wedge g_{1}^{(n-a)} }
 \Biggr|
 \\
 &=
 \ell_a^\rho( c_i)  +
 \log \,
\Biggl|
\frac
{e^{(a-1)}\wedge f_{1}^{(1)} \wedge g_{1}^{(n-a)}}
  {  e^{(a)} \wedge  f_{1}^{(1)} \wedge g_{1}^{(n-a-1)}}
  \,
\frac  { e^{(a+1)} \wedge  g_{1}^{(n-a-1)} }
{ e^{(a)} \wedge g_{1}^{(n-a)} }
 \Biggr|.
\end{align*}

Combining with our earlier computation, this yields
\begin{align*}
\sum_{l=1}^k
\overline
\sigma_a^\rho ( g_{i_l})
&=  \ell_a^\rho( c_i) 
 +
 \sum_{l=1}^k
  \log \,
\Biggl|
\frac
{ e^{(a)}\wedge f_l^{(n-a-1)} \wedge g_l^{(1)}   }
{ e^{(a-1)}\wedge f_l^{(n-a)} \wedge g_l^{(1)} }
\,\frac
{ e^{(a)}\wedge f_{l}^{(n-a)} }
{ e^{(a+1)}\wedge f_{l}^{(n-a-1)} }
  \\
  &\qquad\qquad\qquad\qquad \qquad\qquad
\frac
{  e^{(a-1)}\wedge f_{l}^{(1)} \wedge g_{l}^{(n-a)} }
  {  e^{(a)}\wedge f_{l}^{(1)} \wedge g_{l}^{(n-a-1)} }
 \, \frac
  { e^{(a+1)}\wedge g_{l}^{(n-a-1)} }
{ e^{(a)}\wedge g_{l}^{(n-a)}  }
   \Biggr|
  \\
 &= \ell_a^\rho( c_i)  - \sum_{l=1}^k \log Q_a (E, F_l, G_l)
 = \ell_a^\rho( c_i)  - \sum_{l=1}^k \sum_{b+c=n-a} \kern -4pt \log T_{abc} (E, F_l, G_l)\\
 &= \ell_a^\rho( c_i)  - \sum_{l=1}^k \sum_{b+c=n-a} \tau_{abc}^\rho (T_{j_l}, v_l)
 \end{align*}
 where $Q_a(E, F_l, G_l)$ is the quadruple ratio of \S \ref{subsect:QuadRatios}, and where we use Lemma~\ref{lem:QuadTripleRatios} for the third equality. 
 
 This concludes the proof of Proposition~\ref{prop:LengthFunctionOtherInvariants} in the first case, when the side of $ c_i$ considered is the right-hand side $ c_i^\Right$ and where the  $ g_{i_l}$ spiral on this side in the direction of the orientation of $ c_i$.
 
 Let us now consider the second case, where we are still considering the right-hand side $ c_i^\Right$ but where the spiraling occurs in the direction opposite to the orientation of $ c_i$. The arguments are the same except that
 $$
E^{(a)} = \bigoplus_{b=n-a+1}^n L_b. 
$$
 because the flag $E= \mathcal F_\rho(x)$ is now associated with the negative end point $x$ of $\widetilde c_i$, and is therefore the unstable flag of $ \rho\bigl( [c_i] \bigr)$. It follows that
 $$
 \rho\bigl( [c_i] \bigr)'_*\bigl(e^{(a)}\bigr)=\biggl( \prod_{b=n-a+1}^{n} m_b^\rho( c_i)\biggr)   e^{(a)}
$$
for every  $ e^{(a)} \in \Lambda^{a}(E^{(a)})$.  
Adapting the earlier computation to this case now  leads to the conclusion that 
$$
\sum_{l=1}^k
\overline
\sigma_a^\rho ( g_{i_l})
=
- \ell_{n-a}^\rho( c_i) - \sum_{l=1}^k \sum_{b+c=n-a} \tau_{abc}^\rho (T_{j_l}, v_l).
$$
Replacing $a$ by $n-a$ then gives 
$$
 \ell_a^\rho( c_i) 
 =- \sum_{l=1}^k
\overline
\sigma_{n-a}^\rho ( g_{i_l})
- \sum_{l=1}^k \sum_{b+c=a} \tau_{(n-a)bc}^\rho (T_{j_l}, v_l)
$$
as desired.

The remaining two cases of Proposition~\ref{prop:LengthFunctionOtherInvariants}, where the side of $ c_i$ considered is the left-hand side $ c_i^\Left$ are  obtained from these first two by reversing the orientation of $ c_i$ and using the fact that this replaces $\ell_a^\rho( c_i)$ by $\ell_{n-a}^\rho( c_i)$. 
 \end{proof}

\begin{rem}
\label{rem:ComputeEigenvalues}
In the above proof of Proposition~\ref{prop:LengthFunctionOtherInvariants}, it was convenient to use absolute values everywhere so that we did not have to worry about the signs of the quantities inside logarithms. Another method would have been to take the exponential of the two sides of each equation. This provides a slightly stronger result. For instance, in the first case considered in the proof, this sequence of equations involving exponentials gives
$$
\frac{m_a^\rho(c_i)}{m_{a+1}^\rho(c_i)}
= \ 
\prod_{l=1}^k \exp \overline\sigma_a(g_{i_l}) \ 
\prod_{l=1}^k \prod_{b+c=n-a} \exp \tau_{abc}^\rho(T_{j_l}, v_l)
$$
and directly proves that the quotient on the left-hand side of the equation is positive, without having to rely on Lemma~\ref{lem:HitchinLoxodromicPositiveEigenvalues}. Similar positivity conclusions hold in the other three cases of the proof. 

We will take advantage of this observation in 
\S \ref{subsect:PhiHomeo}.
\end{rem}

\section{Parametrizing the Hitchin component}
\label{sect:ParametrizingHitchin}

\subsection{The space of possible invariants}
\label{subsect:PossibleInvariants}


Recall that we are given a maximal geodesic lamination $\lambda$ with finitely many leaves, consisting of closed leaves $ c_i$ and of infinite leaves $ g_j$ whose ends spiral around the $ c_i$. In addition, each leaf $ c_i$ or $ g_j$ carries an orientation, and each closed leaf $ c_i$ is endowed with an arc $k_i$ transverse to $\lambda$, cutting $ c_i$ in exactly one point, and meeting no other closed leaf $ c_j$ with $j\ne i$. 

We have associated to a Hitchin representation $\rho \colon \pi_1(S) \to \PSL$ triangle invariants $\tau_{abc}^\rho(T_j, v_j)\in \R$ and shear invariants $\sigma_a^\rho( g_j)$ and $\sigma_a^\rho( c_i)$. Lemma~\ref{lem:RotateTirangleInvariants} and Proposition~\ref{prop:LengthFunctionOtherInvariants} provide relations between these invariants.  Let $\mathcal P$ be the set of all possible such invariants. 

More precisely, let $\mathcal P$ be the space of functions  $\tau_{abc}$ and $\sigma_a$  such that:
\begin{enumerate}
\item for every triple of integers $a$, $b$, $c\geq1$ with $a+b+c=n$, $\tau_{abc}$ associates a number $\tau_{abc}(T_j, v_j)\in \R$ to each component $T_j$ of the complement $S-\lambda$ and to each vertex $v_j$ of the ideal triangle $T_j$;
\item for each integer $a=1$, $2$, \dots, $n-1$, $\sigma_a$ associates a number $\sigma_a( c_i)$ or $\sigma_a( g_j)\in \R$ to each leaf $ c_i$ or $ g_j$ of $\lambda$; 
 
\item  for each triangle $T_j$ in the complement $S-\lambda$, the functions $\tau_{abc}$ satisfy the \emph{Rotation Condition} stated below;

\item for every closed leaf $ c_i$ and every index $1\leq a \leq n-1$, the functions $\tau_{abc}$ and $\sigma_a$ satisfy the  \emph{Closed Leaf Equality} condition stated below;

\item for every closed leaf $ c_i$ and every index $1\leq a \leq n-1$, the functions $\tau_{abc}$ and $\sigma_a$ satisfy the  \emph{Closed Leaf Inequality} condition stated below.

\end{enumerate}

The Rotation Condition just comes from Lemma~\ref{lem:RotateTirangleInvariants}. 

\medskip
\noindent\textsc{Rotation Condition.}
If the vertices $v_j$ and $v_j'$ of the triangle component $T_j$ of $S-\lambda$ are such that $v_j'$ immediately follows $v_j$ when going clockwise around the boundary of $T_j$, then 
$$
\tau_{abc} (T_j, v_j) = \tau_{bca}(T_j, v_j').
$$
\medskip

The Closed Leaf Equality and Closed Leaf Inequality are directly inspired by the formulas of Proposition~\ref{prop:LengthFunctionOtherInvariants} computing the $\rho$--lengths $\ell_a^\rho(c_i)$ in terms of the triangle and shear invariants.  Let $ g_{i_1}$, $ g_{i_2}$, \dots, $ g_{i_k}$ be the infinite leaves of $\lambda$ that spiral on the right-hand side $ c_i^\Right$ of  $ c_i$. Let $T_{j_1}$, $T_{j_2}$, \dots, $T_{j_k}$ be the components of the complement $S-\lambda$ that spiral on this  side $ c_i^\Right$ and, for each $j_l$, let $v_{l}$ be the vertex of the ideal triangle $T_{j_l}$ towards which the spiraling occurs.  Similarly, let $ g_{i_1'}$, $ g_{i_2'}$, \dots, $ g_{i_{k'}'}$ be the infinite leaves of $\lambda$ that spiral on the left-hand side $ c_i^\Left$, let $T_{j_1'}$, $T_{j_2'}$, \dots, $T_{j_{k'}'}$ be the components of the complement $S-\lambda$ that spiral on $ c_i^\Left$, and let $v_{l}'$ be the vertex of the ideal triangle $T_{j_l'}$ towards which the spiraling occurs. Following Proposition~\ref{prop:LengthFunctionOtherInvariants}, set $\overline\sigma_a( g_{i_l}) = \sigma_a( g_{i_l})$ if the leaf $ g_{i_l}$ is oriented towards $ c_i$ and $\overline \sigma_a( g_{i_l}) = \sigma_{n-a}( g_{i_l})$ otherwise, and define
$$
L_a^ \Right( c_i) =  \sum_{l=1}^k \overline\sigma_{ a}( g_{i_l}) +   \sum_{l=1}^k \sum_{b+c=n-a} \tau_{abc} (T_{j_l}, v_l)
$$ 
 if the spiraling occurs in the direction of the orientation of $c_i$, and 
 $$
L_a^\Right( c_i)= - \sum_{l=1}^k
\overline
\sigma_{n-a} ( g_{i_l})- \sum_{l=1}^k \sum_{b+c=a} \tau_{(n-a)bc} (T_{j_l}, v_l)
$$
if the spiraling occurs in the direction opposite to the orientation of $c_i$.

Similarly define
$$
L_a^\Left( c_i) =   -\sum_{l=1}^{k'} \overline\sigma_{ a}( g_{i_l'}) -   \sum_{l=1}^{k'} \sum_{b+c=n-a} \tau_{abc} (T_{j_l'}, v_l')
$$
 if the spiraling occurs in the direction of the orientation of $c_i$, and 
$$
L_a^\Left( c_i) =
 \sum_{l=1}^{k'}
\overline
\sigma_{n-a} ( g_{i_l'})
+ \sum_{l=1}^{k'} \sum_{b+c=a} \tau_{(n-a)bc} (T_{j_l'}, v_l')
$$
otherwise. 

The numbers $L_a^\Right(c_i)$ and $L_a^\Right(c_i)$ are completely determined by the functions $\tau_{a'b'c'}$ and $\sigma_{a'}$. Proposition~\ref{prop:LengthFunctionOtherInvariants} says that, when $\tau_{a'b'c'} = \tau_{a'b'c'}^\rho$ and $\sigma_{a'}=\sigma_{a'}^\rho$ are associated to a Hitchin representation $\rho \colon \pi_1(S) \to \PSL$ as in \S \ref{sect:HitchinRepsInvariants}, then 
$$
L_a^\Right( c_i) = L_a^\Left( c_i) = \ell_a^\rho(c_i)>0.
$$
This motivates the following conditions. 

\medskip

\noindent\textsc{Closed Leaf Equality.}
With the above notation,
$$
L_a^\Right( c_i) = L_a^\Left( c_i)
$$
for every closed leaf $ c_i$ of $\lambda$, and for every index $1\leq a\leq n-1$. 

\medskip
\noindent\textsc{Closed Leaf Inequality.}
With the above notation,
$$
L_a^\Right( c_i) >0
$$
for every closed leaf $ c_i$ of $\lambda$, and for every index $1\leq a\leq n-1$. 
\medskip

The geodesic lamination $\lambda$ has $s$ closed leaves and $t$ infinite leaves, and its complement $S-\lambda$ consists of $u$ triangles. Also, there are $\frac{(n-1)(n-2)}2 $ triples of integers $a$, $b$, $c\geq1$ with $a+b+c=n$. Therefore, the   functions $\tau_{abc}$ and $\sigma_a$ satisfying Conditions (1) and (2)  form a vector space of dimension
$$
N= 3u{\textstyle\frac{(n-1)(n-2)}2} + (s+t)(n-1) .
$$

\begin{prop}
\label{prop:Polytope}
The set $\mathcal P$ of functions $\tau_{abc}$ and $\sigma_a$ satisfying the above conditions {\upshape (1--5)} is the interior of a convex polytope $\overline {\mathcal P}\subset \R^N$. 
\end{prop}

\begin{proof}
As a subset of $\R^N$, $\mathcal P$ is defined by a finite collection of linear equalities and strict inequalities. 
\end{proof}

Note that the dimension $d$ of $\overline {\mathcal P}$ is strictly less than $N$. When we refer to the interior of this polytope, we of course mean the topological interior of $\overline {\mathcal P}$ in the $d$--dimensional linear subspace that contains it.

We can formally estimate this dimension $d$. The Rotation Condition enables us to avoid the reference to vertices in the functions $\tau_{abc}$. The space of functions $\tau_{abc}$ satisfying the Rotation Condition is therefore $u\frac{(n-1)(n-2)}2 = 2(g-1)(n-1)(n-2)$, as $u=4(g-1)$.
As before, the space of functions $\sigma_a$ has dimension $(s+t)(n-1)$. 
Since there are $s(n-1)$ Closed Leaf Equalities, the expected dimension of $\mathcal P$ should therefore be
$$
2(g-1)(n-1)(n-2) + (s+t)(n-1) - s(n-1) = 2(g-1)(n^2-1)
$$
if we remember that $t=6(g-1)$. This formal computation can be justified by an explicit argument showing the independence of the relations.  However,  the combinatorics involved are somewhat  complicated. We will be content with the observation that this dimension $2(g-1)(n^2-1)$ is also the dimension of the Hitchin component $\Hit(S)$, and prove that $\mathcal P$ is homeomorphic to $\Hit(S)$. 

In \S \ref{sect:HitchinRepsInvariants}, we associated with each Hitchin representation $\rho \in \Hit(S)$ functions  $\tau_{abc}^\rho$ and $\alpha_a^\rho$ as above, and showed that these functions satisfy the Rotation Condition, the Closed Leaf Equalities  and the Closed Leaf Inequalities (in \S \ref{subsect:Lengths} and  \S \ref{sect:RelationsInvariants}). In other words, we constructed a map
$
\Phi\colon \Hit(S) \to \mathcal P
$.

\begin{thm}
\label{thm:Parametrization}
The above map
$$
\Phi\colon \Hit(S) \to \mathcal P 
$$
is a homeomorphism. 
\end{thm}

The proof of Theorem~\ref{thm:Parametrization} will occupy all of the next section \S \ref{subsect:PhiHomeo}.

\subsection{Proof that the map $\Phi\colon \Hit(S) \to \mathcal P$ is a homeomorphism}
\label{subsect:PhiHomeo}

We begin with a small step. 

\begin{lem}
The map $\Phi\colon \Hit(S) \to \mathcal P$ is continuous. 
\end{lem}
\begin{proof} This is an immediate consequence of the Anosov property, which implies that the flag curve $\mathcal F_\rho \colon \partial_\infty \widetilde S \to \Flag$ depends continuously on the Hitchin representation $\rho$. See \cite[\S2]{Lab1}. The reader can also consult \cite[\S2.6]{Gui} and  \cite[\S5.3]{GuiW}.
\end{proof}

We will construct a right inverse $\mathcal P \to \Hit(S)$ for $\Phi$. For this, suppose that we are given functions $\tau_{abc}$ and $\sigma_a$ that satisfy the Rotation Condition, the Closed Leaf Equalities and the Closed Leaf Inequalities, namely that define a point of the polytope $\mathcal P$. We will construct a Hitchin representation $\rho \in \Hit(S)$ whose invariants are exactly these functions $\tau_{abc}$ and $\sigma_a$.

Our strategy will be to reconstruct the flag curve $\mathcal F_\rho$ of \S \ref{subsect:FlagMap}. However, because we do not yet have a Hitchin representation, we will use a weaker version of this flag curve. 

Let $\partial_\infty \widetilde \lambda$ be the subset of the circle of infinity $\partial_\infty \widetilde S $ that consists of the end points of the leaves of the preimage $\widetilde \lambda \subset \widetilde S$ of the maximal geodesic lamination $\lambda$. More generally, if $\widetilde \lambda' \subset \widetilde \lambda$ is a family of leaves of $\widetilde \lambda$, then  $\partial_\infty \widetilde \lambda' \subset \partial_\infty \widetilde \lambda$ consists of the end points of leaves of $\widetilde\lambda'$. A \emph{flag decoration} for $\widetilde \lambda' $  is a (not necessarily continuous) map $\mathcal F \colon \partial_\infty \widetilde \lambda' \to \Flag$. 

A fundamental example of such a flag decoration of course comes from the restriction to $\partial_\infty \widetilde \lambda$ of the flag curve $\mathcal F_\rho\colon \partial_\infty \widetilde S \to \Flag$ of a Hitchin representation $\rho \colon \pi_1(S) \to \PSL$. 
Note that the definition of the triangle invariants $\tau_{abc}^\rho(T,v)$ and the shear invariants  $\sigma_a^\rho( g_j)$ and $\sigma_a^\rho(c_i)$  only uses the flag decoration $\mathcal F \colon \partial_\infty \widetilde \lambda \to \Flag$, and not the full flag curve $\mathcal F_\rho$.

We can copy these constructions for a general flag decoration $\mathcal F \colon \partial_\infty \widetilde \lambda' \to \Flag$. For instance, if $\widetilde T_j$ is a triangle component of $\widetilde S - \widetilde \lambda$ whose three sides are in the sublamination $\widetilde \lambda'$, and if $\widetilde v_j$ is a vertex of $\widetilde T_j$, we can often define triangle invariants
$$
\tau_{abc}^{\mathcal F} (\widetilde T_j, \widetilde  v_j) = \log T_{abc}
\bigl( \mathcal F(\widetilde v_j), \mathcal F(\widetilde v_j') , \mathcal F(\widetilde v_j'')\bigr)
$$ 
as in \S \ref{subsect:TriangleInvariants}, where the vertices $\widetilde v_j$, $\widetilde v_j'$,  $\widetilde v_j''$ of $\widetilde T$ are indexed in clockwise order. Of course, this triangle invariant only makes sense under the assumption that the triple ratio $T_{abc}\bigl( \mathcal F(\widetilde v_j), \mathcal F(\widetilde v_j') , \mathcal F(\widetilde v_j'')\bigr)$ is positive. 

Similarly, suppose that we are given an isolated leaf $\widetilde  g_i$ of $\widetilde \lambda$ such that  the two triangle components of $\widetilde S - \widetilde \lambda$ that are adjacent to $\widetilde \lambda$ have all their sides contained in $\widetilde\lambda'$. We can then define 
$$
\sigma_a^{\mathcal F}(\widetilde  g_i) = \log D_a(E, F, G, G') = \log D_a \bigl( \mathcal F(x), \mathcal F(y), \mathcal F(z), \mathcal F(z') \bigr)
$$
with the conventions of \S \ref{subsect:InfiniteLeafInvariants}. Again, this requires the double ratio $D_a(E, F, G, G')$ to be positive. 

Finally, let $\widetilde k_i$ be an arc lifting one of the arcs $k_i$ that are part of the topological data of \S \ref{subsect:TopData}. We then define
$$
\sigma_a^{\mathcal F}(\widetilde k_i) = \log D_a(E, F, G, G') = \log D_a \bigl( \mathcal F(x), \mathcal F(y), \mathcal F(z), \mathcal F(z') \bigr)
$$
with the conventions of \S \ref{subsect:ClosedLeafInvariant}, when the points $x$, $y$, $z$, $z'$ are in $\partial _\infty \widetilde \lambda'$ and when the double product $D_a(E, F, G, G')$ is positive.

Note that, without any equivariance condition for the flag decoration $\mathcal F$, there is no reason for the invariants $\tau_{abc}^{\mathcal F} (\widetilde T_j, \widetilde  v_j) $, $\sigma_a^{\mathcal F}(\widetilde  g_i)$ and $\sigma_a^{\mathcal F}(\widetilde k_i)$ to be invariant under the action of $\pi_1(S)$ on $\widetilde S$. 

After these preliminaries on flag decorations, we return to our construction of an inverse for the map $\Phi\colon \Hit(S) \to \mathcal P$. Consider functions $\tau_{abc}$ and $\sigma_a$ that define a point of the polytope $\mathcal P$, namely that satisfy the Rotation Conditions, the Closed Leaf Equalities, and the Closed Leaf Inequalities. 

These functions lift to the universal covering $\widetilde S$, and define numbers $\tau_{abc}(\widetilde T_j, \widetilde v_j)$, $\sigma_a(\widetilde  g_i)$ and $\sigma_a(\widetilde k_i)\in \R$ for every triangle component $\widetilde T_j$ of $\widetilde S-\widetilde \lambda$, every vertex $\widetilde v_j$ of $\widetilde T_j$, every isolated leaf $\widetilde g_i$ of $\widetilde \lambda$, and every arc $\widetilde k_i$ lifting one of the transverse arcs $k_i$. 

\begin{lem}
\label{lem:FlagMapTriangles}
For every component $\widetilde T_j$ of $\widetilde S - \widetilde \lambda$, there exists a flag decoration $\mathcal F\colon \partial _\infty \widetilde T_j \to \Flag$ for the boundary of $\widetilde T_j$ such that
$$
\tau_{abc}^{\mathcal F} (\widetilde T_j, \widetilde v_j) = \tau_{abc}(\widetilde T_j, \widetilde v_j)
$$
for every vertex $\widetilde v_j$ of $\widetilde T_j$ and every integers $a$, $b$, $c\geq1$ such that $a+b+c=n$ (and where $\partial _\infty \widetilde T_j$ consists of the three vertices of $\widetilde T_j$). 

In addition, $\mathcal F$ is unique up to composition with a map $\Flag \to \Flag$ induced by an element of $\PGL$. 
\end{lem}
\begin{proof}
For a single vertex $\widetilde v_j$, this is just another way of saying that there exists a flag triple whose triple ratios are $\exp \tau_{abc}(\widetilde T_j, \widetilde v_j)$, as asserted by Proposition~\ref{prop:TripRatiosDetermineFlagTriples}. The property for all three vertices of $\widetilde T_j$ then follows from the Rotation Condition. The uniqueness property is also a consequence of Proposition~\ref{prop:TripRatiosDetermineFlagTriples}.
\end{proof}

We now put two adjacent triangles together. 

\begin{lem}
\label{lem:FlagMapAdjacentTriangles}
Let $\widetilde T_j$ and $\widetilde T_{j'}'$ be two adjacent components of  $\widetilde S - \widetilde \lambda$, separated by a leaf $\widetilde  g_i$ of $\widetilde\lambda$. Then, if $\widetilde\lambda'$ denotes the union of the sides of $\widetilde T_j$ and $\widetilde T_{j'}'$, there exists a flag decoration 
$\mathcal F\colon \partial _\infty \widetilde \lambda' \to \Flag$  such that
\begin{enumerate}
\item $
\tau_{abc}^{\mathcal F} (\widetilde T_j, \widetilde v) = \tau_{abc}(\widetilde T_j, \widetilde v)
$
for every vertex $\widetilde v$ of $\widetilde T_j$ and every integers $a$, $b$, $c\geq1$ such that $a+b+c=n$;
\item
$
\tau_{abc}^{\mathcal F} (\widetilde T_{j'}', \widetilde v') = \tau_{abc}(\widetilde T_{j'}', \widetilde v')
$
for every vertex $\widetilde v'$ of $\widetilde T_{j'}'$ and every integers $a$, $b$, $c\geq1$ such that $a+b+c=n$;
\item $\sigma_a^{\mathcal F}(\widetilde  g_i) =\sigma_a(\widetilde  g_i) $ for every integer $a$ with $1\leq a \leq n-1$. 
\end{enumerate}

In addition, the flag decoration $\mathcal F\colon \partial _\infty \widetilde \lambda' \to \Flag$  is unique up to post-composition with the action of an element of $\PGL$ on $\Flag$. 
\end{lem}

\begin{proof} Let $\mathcal F\colon \partial _\infty \widetilde T_j \to \Flag$ and $\mathcal F'\colon \partial _\infty \widetilde T_{j'}' \to \Flag$ be given by Lemma~\ref{lem:FlagMapTriangles}. We only need to arrange that $\mathcal F$ and $\mathcal F'$ coincide on the end points of $\widetilde  g_i$, and that Condition~(3) is satisfied. 

Let $x$ and $y$ be the positive and negative end points of $\widetilde g_i$ and, without loss of generality, assume that $\widetilde T_j$ is on the left of $\widetilde g_i$ for the orientation of this leaf. Consider the flags $E= \mathcal F(x)$, $F= \mathcal F(y)$, $E'= \mathcal F'(x)$, $F'= \mathcal F'(y)\in \Flag$.

Because the triangle invariants $\tau_{abc}^{\mathcal F} (\widetilde T_{j'}, \widetilde v)$ are defined, the flag triple $(E,F,G)$ associated by $\mathcal F$ to the three vertices of $\widetilde T_j$ is positive. In particular, the flag pair $(E,F)$ is generic. Similarly, the flag pair $(E', F')$ is also generic. 

Therefore, by elementary linear algebra, there exists an element $A\in \PGL$ that sends $E'$ to $E$ and $F'$ to $F$. As a first approximation, we can then define the flag decoration $\mathcal F\colon \partial _\infty \widetilde \lambda' \to \Flag$ to coincide with the original $\mathcal F$ on $ \partial _\infty \widetilde T_j$, and with $A\circ \mathcal F'$ on $ \partial _\infty \widetilde T_{j'}' $. 

If $z$ and $z'$ are the third vertices of $\widetilde T_j$ and $\widetilde T_{j'}'$, respectively, and if we consider the flags $G= \mathcal F(z)$ and $G' = \mathcal F(z')= A \circ \mathcal F'(z')$, then
$$
\sigma_a^{\mathcal F}(\widetilde  g_i) = \log D_a(E,F,G,G').$$
Of course, at this point, there is no guarantee that the double ratio $D_a(E,F,G,G')$  is positive, so that $\sigma_a^{\mathcal F}(\widetilde  g_i) $ does not necessarily make sense. 

We first compute the double ratio $D_a(E,F,G,G')$ more explicitly. Choose a basis $e_1$, $e_2$, \dots, $e_n$ for $\R^n$ such that each $e_a$ generates the line $E^{(a)} \cap F^{(n-a+1)}$. Express generators for the lines  $G^{(1)}$ and $G'{}^{(1)}$ as  $g_1 = \sum_{a=1}^n \gamma_a e_a$ and $g_1' = \sum_{a=1}^n \gamma_a' e_a$, respectively. Then, if we use the non-zero elements $e^{(b)} = e_1\wedge e_2 \wedge \dots \wedge  e_b \in \Lambda^b(E^{(b)})$ and $f^{(b)} = e_{n-b+1}\wedge e_{n-b+2} \wedge \dots \wedge  e_n \in \Lambda^b(F^{(b)})$ in our computation,
\begin{align*}
D_a(E,F,G,G') &=  - \frac
{ e^{(a)} \wedge  f^{(n-a-1)}\wedge  g_1}
{ e^{(a)} \wedge  f^{(n-a-1)}\wedge  g_1'}\ 
\frac
{ e^{(a-1)} \wedge  f^{(n-a)}\wedge  g_1'}
{ e^{(a-1)} \wedge  f^{(n-a)}\wedge  g_1}\\
&= -\frac{\gamma_{a+1}}{\gamma_{a+1}'} \frac{\gamma_{a}'}{\gamma_{a}}.
\end{align*}

We will now take advantage of the fact that there were many possible choices for $A\in \PGL$. Indeed, we can always post-compose $A$ with an element $B\in \PGL$ respecting the generic flag pair $(E,F)$. Such a $B\in \PGL$ respects each of the lines $ E^{(a)} \cap F^{(n-a+1)}$, and consequently can be represented by a matrix of $\GL$ which acts on each $L_a$ by multiplication by $\beta_a \in \R^*$. Replacing $A$ by $B\circ A$ then replaces each $\gamma_b'$ by $\beta_b \gamma_b'$. We can therefore adjust the matrix $A\in \PGL$ so that 
$$
D_a(E,F,G,G') = \exp \sigma_a(\widetilde  g_i) 
$$
for every integer $a$ with $1\leq a \leq n-1$, and therefore so that $\sigma_a^{\mathcal F}(\widetilde  g_i)  = \sigma_a(\widetilde  g_i) $ for all $a$. 

For such a choice of $A$, the flag decoration $\mathcal F\colon \partial _\infty \widetilde \lambda' \to \Flag$ that coincides with the original $\mathcal F$ on $ \partial _\infty \widetilde T_j$ and with $A\circ \mathcal F'$ on $ \partial _\infty \widetilde T_{j'}' $ then satisfies the conclusions of Lemma~\ref{lem:FlagMapAdjacentTriangles}. 

In fact, in the above argument, the adjustment factor $B$ is unique as an element of $\PGL$. Using the uniqueness in Lemma~\ref{lem:FlagMapTriangles}, the uniqueness part of Lemma~\ref{lem:FlagMapAdjacentTriangles} easily follows. 
\end{proof}

We now put infinitely many adjacent triangles together. Let  $\lambda^{\mathrm{closed}}$ denote the union of the closed leaves of $\lambda$, and let $\widetilde  \lambda^{\mathrm{closed}}$ be its preimage in the universal covering $\widetilde S$. 

\begin{lem}
\label{lem:FlagMapOutsideClosedLeaves}
Let $\widetilde U$ be a component of $\widetilde S - \widetilde\lambda^{\mathrm{closed}}$. Then, there exists a flag decoration $\mathcal F \colon \partial_\infty (\widetilde \lambda \cap \widetilde U) \to \Flag$ such that
\begin{enumerate}
\item $
\tau_{abc}^{\mathcal F} (\widetilde T_j, \widetilde v_j) = \tau_{abc}(\widetilde T_j, \widetilde v_j)
$
for every component $\widetilde T_j$ of $\widetilde U - \widetilde\lambda$, for every vertex $\widetilde v$ of $\widetilde T_j$, and for every integers $a$, $b$, $c\geq1$ such that $a+b+c=n$;

\item $\sigma_a^{\mathcal F}(\widetilde  g_i) =\sigma_a(\widetilde  g_i) $ for every leaf $\widetilde  g_i$ of $\widetilde\lambda \cap \widetilde U$ and for every integer $a$ with $1\leq a \leq n-1$. 
\end{enumerate}

In addition, the flag decoration $\mathcal F \colon \partial_\infty (\widetilde \lambda \cap \widetilde U) \to \Flag$  is unique up to postcomposition with the action of an element of $\PGL$ on $\Flag$. 
\end{lem}

\begin{proof}
By construction, all the leaves of $\widetilde\lambda \cap \widetilde U$ are isolated. Looking at the dual tree of the cell decomposition of $\widetilde U$ induced by $\widetilde\lambda \cap \widetilde U$, we can therefore list all the components of $\widetilde U - \widetilde\lambda$ as $\widetilde T_{j_1}$, $\widetilde T_{j_2}$, \dots, $\widetilde T_{j_k}$, \dots{} in such a way that each $\widetilde T_{j_k}$ is adjacent to exactly one $\widetilde T_{j_l}$ with $l<k$. 

We construct $\mathcal F$  on the closure $\widetilde U_k$ of $\widetilde T_{j_1} \cup \widetilde T_{j_2}\cup \dots \cup \widetilde T_{j_k}$, by induction on $k$.  The induction starts with Lemma~\ref{lem:FlagMapTriangles}. 

Suppose that we have constructed a flag decoration $\mathcal F_{k-1} \colon \partial_\infty (\widetilde \lambda \cap \widetilde U_{k-1}) \to \Flag$ with the desired triangle and shear invariants. By hypothesis, the triangle $\widetilde T_{j_k}$ is adjacent to a triangle $\widetilde T_{j_l}$ with $l<k$; namely the closures of  $\widetilde T_{j_k}$ and $\widetilde T_{j_l}$ meet along a leaf $\widetilde g_i$ of $\widetilde\lambda \cap \widetilde U$. 

Apply Lemma~\ref{lem:FlagMapAdjacentTriangles} to the two triangles $\widetilde T_{j_k}$ and $\widetilde T_{j_l}$. This provides a flag decoration $\mathcal F' \colon \partial_\infty \widetilde T_{j_k} \cup \partial_\infty \widetilde T_{j_l} \to \Flag$ whose triangle and shear invariants are as requested. Composing $\mathcal F'$ with an appropriate element of $\PGL$ (using the uniqueness part of Lemma~\ref{lem:FlagMapTriangles}), we can arrange that $\mathcal F'$ coincides with $\mathcal F_{k-1}$ on $ \partial_\infty \widetilde T_{j_l}$. Since $ \partial_\infty (\widetilde \lambda \cap \widetilde U_k) =  \partial_\infty (\widetilde \lambda \cap \widetilde U_{k-1}) \cup \partial_\infty \widetilde T_{j_k}$, we can then define $\mathcal F_k \colon \partial_\infty (\widetilde \lambda \cap \widetilde U_k) \to \Flag$ to coincide with $\mathcal F_{k-1}$ on $\partial_\infty (\widetilde \lambda \cap \widetilde U_{k-1})$ and with $\mathcal F'$ on $\partial_\infty \widetilde T_{j_k}$. Then, this flag decoration for $\widetilde \lambda \cap \widetilde U_k$ has the required triangle and shear invariants, and proves the induction step. 

This provides a family of flag decorations $\mathcal F_k \colon \partial_\infty (\widetilde \lambda \cap \widetilde U_k) \to \Flag$ such that $\mathcal F_k$ coincides with $\mathcal F_l$ on  $\partial_\infty (\widetilde \lambda \cap \widetilde U_l)$ whenever $k>l$. Since $\widetilde U$ is the union of the $\widetilde U_k$, these $\mathcal F_k$ then give a flag decoration $\mathcal F \colon \partial_\infty (\widetilde \lambda \cap \widetilde U) \to \Flag$ with the requested triangle and shear invariants. 

The uniqueness part of the statement easily follows from that of Lemma~\ref{lem:FlagMapAdjacentTriangles}. 
\end{proof}

\begin{lem}
\label{lem:MonodromyFlagMapAwayFromClosedLeaves}
Under the hypotheses and conclusions of Lemma~{\upshape\ref{lem:FlagMapOutsideClosedLeaves}}, let $\pi_1(U)$ be the stabilizer of $\widetilde U$ in $\pi_1(S)$ (corresponding to the fundamental group of the projection $U$ of $\widetilde U$ onto $S$, for appropriate base points). Then the flag decoration $\mathcal F \colon \partial_\infty (\widetilde \lambda \cap \widetilde U) \to \Flag$ is $\rho$--equivariant for a unique homomorphism $\rho \colon \pi_1(U) \to \PGL$.
\end{lem}
\begin{proof}
Consider an element $\gamma \in \pi_1(U)$.

Let $\widetilde T_j$ be one of the components of $\widetilde U - \widetilde \lambda$, projecting onto a component $T_j$ of $S-\lambda$. By construction, for every vectex $\widetilde v_j$ of $\widetilde T_j$, 
$$\tau_{abc}^{\mathcal F}(\gamma\widetilde T_j, \gamma\widetilde v_j) = \tau_{abc}(\gamma\widetilde T_j, \gamma\widetilde v_j) =  \tau_{abc}( T_j,  v_j) = \tau_{abc}(\widetilde T_j,\widetilde v_j) = \tau_{abc}^{\mathcal F}(\widetilde T_j,\widetilde v_j)
$$
where $v_j$ is the vertex of $T_j$ corresponding to $\widetilde v_j$. By the uniqueness part of Lemma~\ref{lem:FlagMapTriangles}, there consequently exists an element $\rho(\gamma) \in \PGL$ sending the positive flag triple associated with $(\gamma\widetilde T_j, \gamma\widetilde v_j) $ by the flag decoration $\mathcal F$. 

Using the Rotation Condition, this element $\rho(\gamma)$ does not depend on the choice of the vertex $\widetilde v_j$. Also, reconstructing $\widetilde U$ one triangle component at a time as in the proof of Lemma~\ref{lem:FlagMapOutsideClosedLeaves}, and applying each time the uniqueness property of Lemma~\ref{lem:FlagMapAdjacentTriangles}, we see that $\rho(\gamma)$  is also independent of the triangle $\widetilde T_j$. 

Since every point of  $\partial_\infty (\widetilde \lambda \cap \widetilde U)$ is a vertex of some triangle component of $\widetilde S - \widetilde\lambda$, it follows that $\mathcal F(\gamma x) = \rho(\gamma) \mathcal F(x)$ for every $x\in  \partial_\infty (\widetilde \lambda \cap \widetilde U)$. 

This defines a map $\rho\colon \pi_1(U) \to \PGL)$, which is easily seen to be a group homomorphism. The above property shows that the flag decoration  $\mathcal F \colon \partial_\infty (\widetilde \lambda \cap \widetilde U) \to \Flag$ is $\rho$--equivariant. 

The uniqueness of $\rho$ is an immediate consequence of the fact that the action of $\PGL$ on the generic flag triples is free. 
\end{proof}

So far, the arguments were essentially those of Fock and Goncharov in \cite{FoG1}. The next step involves a few new twists, and uses the Closed Leaf Equalities and Inequalities in a critical way. 

\begin{lem}
\label{lem:FlagMapGluedClosedLeaves}
Let $\widetilde U_1$ and $\widetilde U_2$ be two adjacent components of $\widetilde S - \widetilde\lambda^{\mathrm{closed}}$, whose closures meet along a component $\widetilde  c_i$ of $ \widetilde\lambda^{\mathrm{closed}}$.  Then, if $\widetilde V$ denotes the union of $\widetilde U_1$, $\widetilde U_2$ and  $\widetilde c_i$, there exists a flag decoration $\mathcal F \colon \partial_\infty (\widetilde \lambda \cap \widetilde V) \to \Flag$ such that
\begin{enumerate}
\item $
\tau_{abc}^{\mathcal F} (\widetilde T_j, \widetilde v_j) = \tau_{abc}(\widetilde T_j, \widetilde v_j)
$
for every component $\widetilde T_j$ of $\widetilde V - \widetilde\lambda$, for every vertex $\widetilde v$ of $\widetilde T_j$, and for every integers $a$, $b$, $c\geq1$ such that $a+b+c=n$;

\item $\sigma_a^{\mathcal F}(\widetilde  g_i) =\sigma_a(\widetilde  g_i) $ for every leaf $\widetilde  g_i$ of $\widetilde\lambda \cap \widetilde U_1$ or $\widetilde\lambda \cap \widetilde U_2$, and for every integer $a$ with $1\leq a \leq n-1$; 

\item if $ c_i$ is the closed leaf of $\lambda$ that is the image of $\widetilde  c_i$, if $k_i$ is the transverse arc cutting $ c_i$ in one point that is part of the topological data, and if the arc $\widetilde k_i$ lifts $k_i$ to $\widetilde S$ and meets $\widetilde  c_i$ in one point, then $\sigma_a^{\mathcal F}(\widetilde k_i) =\sigma_a(\widetilde k_i)  $ for every  integer $a$ with $1\leq a \leq n-1$. 
\end{enumerate}

In addition, the flag decoration $\mathcal F \colon \partial_\infty (\widetilde \lambda \cap \widetilde V) \to \Flag$  is unique up to post-composition with the action of an element of $\PGL$ on $\Flag$. 
\end{lem}

Note that Condition~(3) has to be satisfied for every arc $\widetilde k_i$ lifting $k_i$ as indicated. This condition is much stronger than one could think at first glance. As we will see, it requires that  the Closed Leaf Equalities hold for the functions $\tau_{abc}$ and $\sigma_a$.

\begin{proof}
Let $\mathcal F_1 \colon \partial_\infty (\widetilde \lambda \cap \widetilde U_1) \to \Flag$  and $\mathcal F_2 \colon \partial_\infty (\widetilde \lambda \cap \widetilde U_2) \to \Flag$  be the flag decorations provided by Lemma~\ref{lem:FlagMapOutsideClosedLeaves}, respectively equivariant with respects to homomorphisms $\rho_1 \colon \pi_1(U_1) \to \PGL$ and $\rho_2 \colon \pi_1(U_2) \to \PGL$ as in Lemma~\ref{lem:MonodromyFlagMapAwayFromClosedLeaves}. 

The leaf $\widetilde c_i$ has an infinite cyclic stabilizer in $\pi_1(S)$, generated by the element $ [c_i]\in \pi_1(U_1) \cap \pi_1(U_2)$ defined by the choice of an appropriate path connecting the oriented closed curve $ c_i$ to the base point used in the definition of $\pi_1(S)$.

The computations of \S \ref{sect:RelationsInvariants} determine the eigenvalues of $\rho_1\bigl([ c_i]\bigr)$. More precisely, the leaves of $\widetilde\lambda \cap \widetilde U_1$ that are asymptotic to $\widetilde c_i$ all have one endpoint $u\in\partial_\infty (\widetilde \lambda \cap \widetilde U_1)$  in common, corresponding to the positive or negative end point of $\widetilde c_i$ according to the direction of the spiraling. By construction, $\rho_1( c_i)\in \PSL$ respects the flag $H_1= \mathcal F_1(u)\in \Flag$ associated to $u$ by the flag decoration $\mathcal F_1$. Lift $\rho_1\bigl([ c_i]\bigr)\in \PGL$  to $\rho_1\bigl([ c_i]\bigr)'\in \GL$ . We then consider the eigenvalue $m_a^{\rho_1}( c_i)$ of $\rho_1\bigl([ c_i]\bigr)'$ such that $\rho_1\bigl([ c_i]\bigr)'$ acts by multiplication by $m_a^{\rho_1}( c_i)$ on $H_1^{(a)}/H_1^{(a+1)}\cong \R$ if $u$ is the positive end point  of $\widetilde c_i$, and on $H_1^{(n-a+1)}/H_1^{(n-a)}\cong \R$ if $u$ is the negative end point of $\widetilde c_i$. 

The formulas of  \S \ref{sect:RelationsInvariants} then compute each ratio $\frac{m_a^{\rho_1}( c_i)}{m_{a+1}^{\rho_1}( c_i)}$ in terms of 
the triangle invariants of the components of $\widetilde U_1 - \widetilde \lambda$ and of the shear invariants of the leaves of $\widetilde \lambda \cap \widetilde U_1$. We  cannot quite apply Proposition~\ref{prop:LengthFunctionOtherInvariants} as is, because we do not (yet) know that $\rho_1$ is the restriction of a Hitchin representation. However, as observed in  Remark~\ref{rem:ComputeEigenvalues}, the arguments of the proof of this statement straightforwardly apply to the current case as well. The conclusion is then that, if $L_a^{\Left}(c_i)$ and   $L_a^{\Right}(c_i)$ are defined as for the Closed Leaf Equalities and Inequalities in \S \ref{subsect:PossibleInvariants}, the quotient $\frac{m_a^{\rho_1}( c_i)}{m_{a+1}^{\rho_1}( c_i)}$ is equal to $\exp L_a^\Left(c_i)$ or $\exp L_a^\Right(c_i)$ according to whether $\widetilde U_1$ is to the left or to the right of $\widetilde c_i$ for the orientation of $c_i$.

Since the functions $\tau_{abc}$ and $\sigma_a$ satisfy  the Closed Leaf Inequalities, each of these ratios $\frac{m_a^{\rho_1}( c_i)}{m_{a+1}^{\rho_1}( c_i)} = \exp L_a^{\Left/\Right}(c_i)$ is strictly greater than 1. In particular, the eigenvalues $m_a^{\rho_1}( c_i)$ are all distinct, and $\rho_1( c_i)'$ is diagonalizable. Also,  the eigenspace $L_a$ corresponding to the eigenvalue $m_a^{\rho_1}( c_i)$ is 1--dimensional. Consider the flags $E_1$, $F_1\in \Flag$ defined by the property that $E_1^{(a)} = \sum_{b=1}^a L_a$ and $F_1^{(a)} = \sum_{b=n-a+1}^n L_a$. Note that our original flag $H_1= \mathcal F_1(u)$ is equal to either $E_1$ or $F_1$, according to whether $u$ is the positive or negative end point of $\widetilde  c_i$. 

Switching now to $\widetilde U_2$, the same argument provides two flags $E_2$, $F_2\in \Flag$ invariant under $\rho_2\bigl([ c_i]\bigr)$ and eigenvalues $m_a^{\rho_2}( c_i)>0$ of a lift $\rho_2( c_i)'\in \GL$ of $\rho_2\bigl([ c_i]\bigr)\in \PGL$ such that $\rho_2( c_i)'$ acts by multiplication of $m_a^{\rho_2}( c_i)$ on $E_2^{(a)}/E_2^{(a+1)} \cong F_2^{(n-a+1)}/F_2^{(n-a)} \cong \R$. 

We now use the fact that the functions $\tau_{abc}$ and $\sigma_a$ satisfy the Closed Leaf Equalities associated with the closed leaf $c_i$. This implies that  $\frac{m_a^{\rho_1}( c_i)}{m_{a+1}^{\rho_1}( c_i)} = \frac{m_a^{\rho_2}( c_i)}{m_{a+1}^{\rho_2}( c_i)}$ for every $a$. As a consequence, the lift $\rho_2\bigl([ c_i]\bigr)'\in \GL$ of $\rho_2\bigl([ c_i]\bigr)\in \PGL$ can be chosen so that it has the same eigenvalues as $\rho_1\bigl([ c_i]\bigr)$, and therefore so that it is conjugate to $\rho_1\bigl([ c_i]\bigr)$ by a matrix  $A\in\GL$.

The matrix  $A\in \GL$ sends the eigenspaces of $\rho_2\bigl([ c_i]\bigr)$ to the eigenspaces of $\rho_1\bigl([ c_i]\bigr)$, and the induced map $ \Flag \to \Flag$ therefore sends $E_2$ to $E_1$ and $F_2$ to $F_1$. 

The set  $\partial_\infty (\widetilde \lambda \cap \widetilde V)$ is the union of $ \partial_\infty (\widetilde \lambda \cap \widetilde U_1)$, of $ \partial_\infty (\widetilde \lambda \cap \widetilde U_2)$ and of the two end points of $\widetilde  c_i$. (In fact, $ \partial_\infty (\widetilde \lambda \cap \widetilde U_1)$ already contains one of the end points of $\widetilde  c_i$, and so does $ \partial_\infty (\widetilde \lambda \cap \widetilde U_2)$.) We can therefore define a flag decoration $\mathcal F \colon \partial_\infty (\widetilde \lambda \cap \widetilde V) \to \Flag$ by the property that it coincides with $\mathcal F_1$ on $ \partial_\infty (\widetilde \lambda \cap \widetilde U_1)$, it coincides with $A \circ \mathcal F_2$ on $ \partial_\infty (\widetilde \lambda \cap \widetilde U_2)$, and it sends the positive and negative end points of $\widetilde c_i$ to $E_1=A(E_2)$ and $F_1=A(F_2)\in \Flag$, respectively. 

This flag decoration $\mathcal F \colon \partial_\infty (\widetilde \lambda \cap \widetilde V) \to \Flag$ clearly satisfies Conditions~(1) and (2) of Lemma~\ref{lem:FlagMapGluedClosedLeaves}, since these conditions only involve subsets of $\widetilde U_1$ and $\widetilde U_2$. We now have to worry about  Condition~(3).

Choose a lift  $\widetilde k_i\subset \widetilde S$ of the arc $k_i$ that meets $\widetilde  c_i$ in one point. At this point, we still have a certain amount of flexibility in the construction of the flag decoration $\mathcal F$, since we can replace $A$ by $B\circ A$, where $B$ commutes with $\rho_1\bigl([ c_i]\bigr)$ and  stabilizes the generic flag pair $(E_1, F_1)$. As in the proof of Lemma~\ref{lem:FlagMapAdjacentTriangles}, we can use this flexibility to  guarantee that $\sigma_a^{\mathcal F}(\widetilde k_i) = \sigma_a(\widetilde k_i)$. The proof is essentially identical to the one used in the proof of  Lemma~\ref{lem:FlagMapAdjacentTriangles}, with only minor differences in the notation, so we  will not repeat it here. 

This takes care of one of the lifts $\widetilde k_i$ of the arc $k_i$. If $\widetilde k_i'$ is any other lift of $k_i$ that meets $\widetilde c_i$ in one point, there exists a power $[ c_i]^k$ of $[ c_i]\in \pi_1(S)$ such that $\widetilde k_i'=  [c_i]^k \widetilde k_i$. Note that $ [c_i] \in \pi_1(U_1) \cap \pi_1(U_2)$ respects $\widetilde V$ and $\widetilde \lambda \cap \widetilde V$ and, by construction, the flag decoration $\mathcal F \colon \partial_\infty (\widetilde \lambda \cap \widetilde V) \to \Flag$ is equivariant with respect to $\rho_1\bigl( [ c_i]\bigr) = A\circ  \rho_2\bigl( [c_i] \bigr)\circ  A^{-1}\in \PGL$. Using the fact that the original flag decorations $\mathcal F_1$ and $\mathcal F_2$ are respectively $\rho_1$-- and $\rho_2$--equivariant, it follows that  $\sigma_a^{\mathcal F}(\widetilde k_i') = \sigma_a^{\mathcal F}(\widetilde k_i) = \sigma_a(\widetilde k_i)$. Therefore, Condition~(3) holds. 

As in the proof of Lemma~\ref{lem:FlagMapAdjacentTriangles}, the uniqueness of the flag decoration $\mathcal F$ up to the action of $\PGL$ follows from the uniqueness of $\mathcal F_1$ and $\mathcal F_2$, and from the fact that the adjustment factor $B$ is unique in $\PGL$. 
\end{proof}

We are now ready to construct the full flag decoration $\mathcal F \colon \partial_\infty \widetilde \lambda \to \Flag$ that we need. 
\begin{lem}
\label{lem:FlagMapConstructed}
There exists a flag decoration $\mathcal F \colon \partial_\infty \widetilde \lambda \to \Flag$ such that
\begin{enumerate}
\item $
\tau_{abc}^{\mathcal F} (\widetilde T_j, \widetilde v_j) = \tau_{abc}(\widetilde T_j, \widetilde v_j)
$
for every component $\widetilde T_j$ of $\widetilde S - \widetilde\lambda$, for every vertex $\widetilde v$ of $\widetilde T_j$, and for every integers $a$, $b$, $c\geq1$ such that $a+b+c=n$;

\item $\sigma_a^{\mathcal F}(\widetilde  g_i) =\sigma_a(\widetilde  g_i) $ for every isolated  leaf $\widetilde  g_i$ of $\widetilde\lambda$ and for every integer $a$ with $1\leq a \leq n-1$; 

\item  $\sigma_a^{\mathcal F}(\widetilde k_i) =\sigma_a(\widetilde k_i)  $ for every arc  $\widetilde k_i\subset \widetilde S$ lifting one of the transverse arcs $k_i$ that are part of the topological data, and  for every  integer $a$ with $1\leq a \leq n-1$. 
\end{enumerate}
In addition, $\mathcal F$ is unique up to post-composition by the map $\Flag \to \Flag$ induced by an element of $\PGL$. 
\end{lem}
\begin{proof}
The argument is very similar to the one used in the proof of Lemma~\ref{lem:FlagMapOutsideClosedLeaves}. List the components of $\widetilde S - \widetilde\lambda$ as $\widetilde U_1$, $\widetilde U_2$, \dots, $\widetilde U_k$, \dots{} in such a way that each $\widetilde U_k$ is adjacent to exactly one $\widetilde U_l$ with $l<k$. One then construct $\mathcal F$ on $\widetilde U_1 \cup \widetilde U_2 \cup \dots \cup \widetilde U_k$ by induction on $k$, using Lemma~\ref{lem:FlagMapGluedClosedLeaves} at each stage. 
\end{proof}

\begin{lem}
\label{lem:FlagMapMonodromy}
Under the hypotheses and conclusions of Lemma~{\upshape\ref{lem:FlagMapConstructed}}, there exists a unique homomorphism $\rho \colon \pi_1(S) \to \PGL$ for which the flag decoration $\mathcal F \colon \partial_\infty \widetilde \lambda \to \Flag$ is $\rho$--equivariant. 
\end{lem}

\begin{proof}
The proof is essentially  identical to that of Lemma~\ref{lem:MonodromyFlagMapAwayFromClosedLeaves}, using the uniqueness statements of Lemmas~\ref{lem:FlagMapGluedClosedLeaves} and \ref{lem:FlagMapConstructed}. We omit the details.  
\end{proof}

Before passing to the next step, we note the following elementary fact. 

\begin{lem}
\label{lem:IncompleteFlagTripleUnique}
Given generic flag triples $(E,F,G)$ and $(E', F', G')$, there exists a unique element of $\PGL$ that sends the flag $E$ to the flag $E'$, the flag $F$ to the flag $F'$, and the line $G^{(1)}$ to the line $G'{}^{(1)}$.   \qed
\end{lem}

Lemma~\ref{lem:IncompleteFlagTripleUnique} will enable us to freeze the $\PGL$--ambiguity  in the construction of the flag decoration $\mathcal F \colon \partial_\infty \widetilde  \lambda \to \Flag$ of Lemma~\ref{lem:FlagMapConstructed}, and of the homomorphism $\rho\colon \pi_1(S) \to \PGL$ of Lemma~\ref{lem:FlagMapMonodromy}. 

We begin with an arbitrary Hitchin representation $\rho_0 \colon \pi_1(S) \to \PSL$ and its associated flag curve $\mathcal F_{\rho_0} \colon \partial_\infty \widetilde S \to \Flag$. We also select an arbitrary component $\widetilde T_{i_0}$ of $\widetilde S - \widetilde\lambda$, and let $x_0$, $y_0$, $z_0\in \partial_\infty \widetilde S$ be the vertices of this triangle. Finally, we consider the positive flag triple $(E_0, F_0, G_0)$ where $E_0 = \mathcal F_{\rho_0}(x_0)$,  $F_0 = \mathcal F_{\rho_0}(y_0)$ and  $G_0 = \mathcal F_{\rho_0}(z_0)$. 

By  Lemma~\ref{lem:IncompleteFlagTripleUnique}, the flag decoration $\mathcal F \colon \partial_\infty \widetilde \lambda \to \Flag$ can be modified by an element of $\PGL$  so that $\mathcal F(x_0)=E_0$, $\mathcal F(y_0)=F_0$ and $\mathcal F(z_0)^{(1)}=G_0^{(1)}$. With such a normalization, the flag decoration $\mathcal F$ is now uniquely determined. So is the homomorphism $\rho \colon \pi_1(S) \to \PGL$.

\begin{lem}
\label{lem:FlagMapMonodromyIsHitchin}
With the above normalization, the homomorphism  $\rho$ of Lemma~{\upshape\ref{lem:FlagMapMonodromy}} is valued in $\PSL$ and is a Hitchin representation. 
\end{lem}

Without the above normalization, we would only conclude that $\rho$ is conjugate to a Hitchin representation by an element of $\PGL$. This does not make any difference when $n$ is odd, since $\PGL = \PSL$ in this case. However, when $n$ is even, the quotient $\PGL/\PSL=\Z_2$ acts by conjugation on the character variety $\mathcal R_{\PSL}(S)$, and sends the Hitchin component $\Hit(S)$ to a different component.  

 \begin{proof} We will use a continuity and connectedness argument. 
 
If we examine the proofs of Lemmas~\ref{lem:FlagMapTriangles}--\ref{lem:FlagMapConstructed} that lead to the construction of the flag decoration $\mathcal F \colon \partial_\infty \widetilde \lambda \to \Flag$, we see that $\mathcal F$ depends continuously on the functions $\tau_{abc}$ and $\sigma_a$. More precisely, if $\mathcal F \colon \partial_\infty \widetilde \lambda \to \Flag$ is normalized as above, then for every $u \in  \partial_\infty \widetilde \lambda$ the flag $\mathcal F(u)\in \Flag$ depends continuously on the finitely many parameters $\tau_{abc}(T_j, v_j)$, $\sigma_a(g_i)$, $\sigma(c_i)\in \R$. 

The normalization of  the flag decoration $\mathcal F$ provides a normalization of the homomorphism $\rho \colon \pi_1(S) \to \PGL$ for which $\mathcal F$ is $\rho$--equivariant. Indeed, if $x_0$, $y_0$, $z_0\in \partial_\infty \widetilde S$ are the vertices of the base triangle $\widetilde T_{i_0}$ that we have chosen and if $\gamma \in \pi_1(S)$, $\rho(\gamma)$ is the unique element of $\PGL$ sending $\mathcal F(x_0)$ to $\mathcal F(\gamma x_0)$, $\mathcal F(y_0)$ to $\mathcal F(\gamma y_0)$ and $\mathcal F(z_0)$ to $\mathcal F(\gamma z_0)$ in $\Flag$. In particular, this proves that $\rho$ depends continuously on  the functions $\tau_{abc}$ and $\sigma_a$. 

If $\mathcal P$ is the polytope of Proposition~\ref{prop:Polytope},  consisting of all functions $\tau_{abc}$ and $\sigma_a$ satisfying the Rotation Condition, the Closed Leaf Equalities and the Closed Leaf Inequalities, we consequently have constructed a continuous map
$$
\Psi \colon \mathcal P \to \{ \text{homomorphisms } \rho \colon \pi_1(S) \to \PGL \}.
$$

Now, let us return to the  Hitchin representation $\rho_0 \colon \pi_1(S)\to \PSL$ used to normalize the flag decoration $\mathcal F$ and the homomorphism $\rho$. 
The triangle and shear invariants $\tau_{abc}^{\rho_0}$ and $\sigma_a^{\rho_0}$ of $\rho_0$ define a point $P_0$ of the polytope $\mathcal P$. Because our choice of normalization is specially taylored for $\rho_0$,  the normalized  flag decoration $\mathcal F \colon \partial_\infty \widetilde \lambda \to \Flag$ associated with $P_0\in \mathcal P$ coincides with the restriction of the flag curve $\mathcal F_{\rho_0}$ to $\partial_\infty \widetilde \lambda$. As a consequence, $\Psi(P_0) = \rho_0$. 

Therefore, there is at least one point $P_0\in \mathcal P$ such that the homomorphism $\rho_0 = \Psi(P_0)$ is valued in $\PSL$ (and not just in $\PGL$) and is a Hitchin representation. Because  the convex polytope $\mathcal P$ is connected, we conclude by continuity that for every $P\in \mathcal P$ the homomorphism $\Psi(P)$ is valued in $\PSL$. Also, the Hitchin representations form a whole component of $\{ \text{homomorphisms } \rho \colon \pi_1(S) \to \PSL \}$. The same  continuity and connexity argument then shows that every  $\Psi(P)$  is a Hitchin representation. 
\end{proof}

Composing $\Psi$ with the quotient map under the action of $\PSL$ therefore provides a continuous map
$$
\overline \Psi \colon \mathcal P \to \Hit(S)
$$
 from the convex polytope $P$ to the Hitchin component $\Hit(S)$. 

\begin{lem}
\label{lem:ParametrizationInverse}
The map $\overline \Psi$ is an inverse of $\Phi\colon \Hit(S) \to \mathcal P$. 
\end{lem}

\begin{proof}
Let $P$ be a point of $\mathcal P$, consisting of functions $\tau_{abc}$ and $\sigma_a$.  We just proved that $\rho = \Psi(P)$ is a Hitchin representation. Let $\mathcal F_\rho \colon \partial_\infty \widetilde S \to \Flag$ be its flag curve. 

Let  $\mathcal F \colon  \partial_\infty \widetilde \lambda \to \Flag$ be the normalized flag decoration used in the construction of $\rho$. Every point  $ x\in \partial_\infty \widetilde \lambda $ is an end point of a component $\widetilde c_i$ of a closed leaf $c_i$ of $\lambda$, and this component $\widetilde c_i$ is invariant under an element $[c_i] \in \pi_1(S)$ represented by $c_i$ suitably connected to the base point by a path. By construction (see the proof of Lemma~\ref{lem:FlagMapGluedClosedLeaves}), $\mathcal F(x)$ is the stable or unstable flag of $\rho \bigl( [c_i] \bigr)$, according to whether $x$ is the positive or negative endpoint of $\widetilde c_i$. Consequently, the flag decoration $\mathcal F$ is just the restriction of the flag curve $\mathcal F_\rho$ to $\partial_\infty \widetilde \lambda $. 

It follows that the triangle and shear invariants $\tau_{abc}^\rho$ and $\sigma_a^\rho$ of $\rho$ are equal to the triangle and shear invariants of the flag decoration $\mathcal F$, namely are equal to the functions $\tau_{abc}$ and $\sigma_a$ we started with. This can be rephrased as $\Phi \bigl( \overline \Psi(P) \bigr) =P$. 

	Since this holds for every $P \in \mathcal P$, this proves that $\Phi \circ \overline \Psi = \Id_{\mathcal P}$. 
	
	Conversely, let $[\rho] \in \Hit(S)$ be represented by a Hitchin representation $\rho \colon \pi_1(S) \to \PSL$. The image $P=\Phi\bigl( [\rho ] \bigr)\in \mathcal P$ is defined by the triangle and shear invariants $\tau_{abc}^\rho$ and $\sigma_z^\rho$ of $\rho$. 
	
	To determine $\rho' = \Psi(P)$, we need a normalized flag decoration whose triangle and shear invariants correspond to $P$, namely are equal to $\tau_{abc}^\rho$ and $\sigma_a^\rho$. The restriction to $\partial_\infty \widetilde\lambda$ of the flag curve $\mathcal F_\rho \colon \partial_\infty \widetilde S \to \Flag$ has the correct invariants, but is not normalized. If $x_0$, $y_0$, $z_0$ are the vertices of the component $\widetilde T_{i_0}$ of $\widetilde S - \widetilde \lambda$ that we used as a base triangle in the normalization, Lemma~\ref{lem:IncompleteFlagTripleUnique} says that there exists a unique $A_\rho \in \PGL$ that sends the flag $\mathcal F_\rho(x)$ to $E_0$,  the flag $\mathcal F_\rho(y)$ to $F_0$ and the line   $\mathcal F_\rho(z)^{(1)}$ to $G_0^{(1)}$. When $\rho$ is the Hitchin representation $\rho_0$ used to define the flag triple $(E_0, F_0. G_0)$ in the normalization, $A_{\rho_0}$ is equal to the identity. Since $A_\rho$ depends continuously on $\rho$, it is therefore contained in the component of the identity in $\PGL$, namely in $\PSL$. 
	
	Now, the restriction of $A_\rho \circ \mathcal F_{\rho}$ to $\partial_\infty \widetilde\lambda$ is a normalized decoration  whose triangle and shear invariants correspond to $P\in \mathcal P$. It is equivariant with respect to the homomorphism $\rho'$ obtained by conjugating $\rho$ with $A_\rho$, so that the Hitchin representation $\Psi(P) $ is equal to $\rho'$. Since we just proved that $A_\rho$ is in $\PSL$, the homomorphism $\rho'$ represents the same element of the character variety $\mathcal R_{\PSL}(S)$ as $\rho$, and $\overline \Psi(P) = [\rho'] = [\rho] \in \Hit(S)$. 		
			
	This proves that $\overline \Psi \bigl( \Phi([\rho]) \bigr) = [\rho]$ for every $[\rho]\in \Hit(S)$, namely that $\overline \Psi \circ \Phi = \Id_{\Hit(S)}$. 			
	Since $\overline \Psi \circ \Phi = \Id_{\Hit(S)}$ and $\Phi \circ \overline \Psi = \Id_{\mathcal P}$, this proves that $\overline\Psi$ is the inverse of $\Phi$. 
\end{proof}

This proves that $\Phi \colon \Hit(S) \to \mathcal P$ is a homeomorphism.

The following proposition  completes the proof of Theorem~\ref{thm:Parametrization}. 

\begin{prop}
The map $\Phi \colon \Hit(S) \to \mathcal P$ and its inverse $\overline \Psi$ are real analytic. 
\end{prop}

\begin{proof}
For a Hitchin representation $\rho$, the triangle invariant $\tau_{abc}^\rho(T_j, v_j)$ is a (real) analytic function of the three flags $\mathcal F_\rho(x)$, $\mathcal F_\rho(y)$, $\mathcal F_\rho(z)\in \Flag$ associated by the flag curve $\mathcal F_\rho$ to the vertices $x$, $y$, $z\in \partial_\infty \widetilde S$ of a lift $\widetilde T_i$ of the triangle $T_i$. 

Because the three ends of $T_j$ spiral around closed leaves $c_i$ of the geodesic lamination $\lambda$, the vertex $x\in \partial_\infty \widetilde S$ is the stable or unstable fixed point of some element $[c_i] \in \pi_1(S)$ represented by one of the closed leaves $c_i$. Therefore, by Part~(1) of Proposition~\ref{prop:FlagMap}, $\mathcal F_\rho(x)$ is the stable or unstable flag of $\rho \bigl( [c_i] \bigr)\in \PSL$. Since the matrix $\rho \bigl( [c_i] \bigr)$ is an analytic function of $\rho$, so is its stable or unstable flag $\mathcal F_\rho(x)$. The same of course holds for $\mathcal F_\rho(y)$ and  $\mathcal F_\rho(z)$.

This proves that the three flags $\mathcal F_\rho(x)$, $\mathcal F_\rho(y)$, $\mathcal F_\rho(z)$  analytically depend on the representation $\rho$. It follows that the triangle invariant $\tau_{abc}^\rho(T_j, v_j)$ is an analytic function of $\rho$. 

The same argument shows that each shear invariant $\sigma_a(g_i)$ or $\sigma_a(c_j)$ is also an analytic function of $\rho$. 

This proves that the point $\Phi\bigl([\rho]\bigr)\in \mathcal P$ represented by the triangle and shear invariants of the Hitchin representation $\rho$ analytically depends on $\rho$. In other words, $\Phi \colon \Hit(S) \to \mathcal P$ is analytic. 

Conversely, consider our definition of the inverse map $\overline \Psi=\Phi^{-1} \colon \mathcal P \to \Hit(S)$. Given a point $P\in \mathcal P$ represented by functions $\tau_{abc}$ and $\sigma_a$, we constructed a normalized flag decoration $\mathcal F \colon \partial_\infty \widetilde\lambda \to \Flag$. This construction, developed in the proofs of Lemmas~\ref{lem:FlagMapTriangles}--\ref{lem:FlagMapConstructed}, is very explicit. As a consequence, for every $x\in \partial_\infty \widetilde\lambda$, the flag $\mathcal F(x) \in \Flag$ is an analytic function of the point $P\in \mathcal P$. 

We now consider the Hitchin representation $\rho = \Psi(P)$ with respect to which $\mathcal F$ is $\rho$--equivariant. Let $x$, $y$, $z\in \partial_\infty \widetilde\lambda$ be the vertices of a fixed triangle component $\widetilde T_j$ of $\widetilde S - \widetilde \lambda$. Then, for each $\gamma \in \pi_1(S)$, the element $\rho(\gamma) \in \PSL$ can be analytically expressed in terms of the six flags $\mathcal F_\rho(x)$, $\mathcal F_\rho(y)$, $\mathcal F_\rho(z)$, $\mathcal F_\rho(\gamma x)$, $\mathcal F_\rho(\gamma  y)$, $\mathcal F_\rho(\gamma  z)\in \Flag$. As a consequence, $\rho(\gamma)$ analytically depends on the point $P \in \mathcal P$. This proves that $\rho= \Psi(P)$ is an analytic function of $P$. 

In other words,  the function 
$$
\Psi \colon \mathcal P \to \{ \text{homomorphisms } \rho \colon \pi_1(S) \to \PGL \}
$$
is analytic. Its composition $\overline \Psi$ with the projection to $\Hit(S)$ is therefore analytic. 
\end{proof}

\section{Global relations between triangle invariants}

Compared to the classical case of the parametrization of the Teichm\"uller space $\mathcal F(S)$ by shear coordinates, the really new feature in the parametrization of Theorem~\ref{thm:Parametrization} is provided by the triangle invariants $\tau_{abc}(T_j, v_j)$. A somewhat surprising property of these triangle invariants is that they are not independent of each other, and are constrained by certain linear relations. 

\begin{prop}
\label{prop:RelationsTriangleInvariants}
Let $\rho$ be a Hitchin representation with triangle invariants $\tau_{abc}^\rho(T_j, v_j)$. Then,  for every integer $a$ with $1 \leq a \leq n-1$, 
$$
 \sum_{j=1}^u \ \sum_{v_j \text{\,vertex\,of } T_j} \biggl(
 \sum_{b+c=n-a} \tau_{abc}^\rho(T_j, v_j) -  \sum_{b+c=a} \tau_{(n-a)bc}^\rho(T_j, v_j) \biggr)=0
$$
where the first sum is over all components $T_1$, $T_2$, \dots, $T_u$ of $S-\lambda$, and where the second sum is over all three vertices of the triangle $T_j$. 
\end{prop}

Note that the equation associated with the index $n-a$ is, up to sign, the same as  the equation associated with $a$. So in practice there are only $\lfloor \frac{n-1}2\rfloor$ equations here. One easily sees that these $\lfloor \frac{n-1}2\rfloor$ equations are linearly independent, as they involve different sets of terms $\tau_{a'b'c'}^\rho(T_j, v_j)$. 

\begin{proof}
This is a consequence of Proposition~\ref{prop:LengthFunctionOtherInvariants}.

We will use a slightly different notation for the formulas of  Proposition~\ref{prop:LengthFunctionOtherInvariants}. Write $g_k \to c_i^\Right$ to indicate that one end of the infinite leaf $g_k$ spirals towards the right-hand side $c_i^\Right$. Similarly, we will write $(T_j, v_j) \to c_j^\Right$ when the triangle component $T_j$ of $S-\lambda$ spirals towards $c_i^\Right$, in the direction of the vertex $v_j$ of $T_j$. When $g_k \to c_i^\Right$, the quantity $\overline\sigma_a^\rho( g_k)$ denotes $ \sigma_a^ \rho( g_k)$ if the leaf $ g_k$ is oriented towards $ c_i$ and $ \sigma_{n-a}^\rho( g_k)$ otherwise. 

Proposition~\ref{prop:LengthFunctionOtherInvariants} computes the length $\ell_a^\rho(c_i)$ in terms of the triangle and shear invariants of the triangles and leaves spiraling on the right-hand side of $c_i$. The corresponding formula depends on the direction of the spiraling on the right-hand side of $c_i$. However, there is no such distinction to be made when computing the difference $\ell_a^\rho(c_i)-\ell_{n-a}^\rho(c_i)$.
Indeed, independently of the direction of the spiraling, 
\begin{align*}
\ell_a^\rho(c_i)-\ell_{n-a}^\rho(c_i) &=
\sum_{g_k \to c_i^\Right} \bigl( \overline\sigma_a^\rho(g_k) -   \overline\sigma_{n-a}^\rho(g_k) \bigr)\\
&\qquad + \sum_{(T_j, v_j) \to c_i^\Right} \biggl( \sum_{b+c=n-a} \tau_{abc}^\rho(T_j, v_j) -  \sum_{b+c=a} \tau_{(n-a)bc}^\rho(T_j, v_j) \biggr).
\end{align*}
Note that an infinite leaf $g_k$ whose two ends spiral towards $c_i^\Right$ will contribute two terms to the first sum, one for each end of $g_k$; there is a definite abuse of notation in this case, as these two contributions are both written as $ \overline\sigma_a^\rho(g_k) -   \overline\sigma_{n-a}^\rho(g_k) $, but are equal to $\sigma_a^\rho(g_k) -  \sigma_{n-a}^\rho(g_k)$ for the positive end and $\sigma_{n-a}^\rho(g_k) -  \sigma_{a}^\rho(g_k)$ for the negative end.

Switching attention to the left-hand side $c_i^\Left$, Proposition~\ref{prop:LengthFunctionOtherInvariants} similarly gives
\begin{align*}
\ell_a^\rho(c_i)-\ell_{n-a}^\rho(c_i) &=
-\sum_{g_k \to c_i^\Left} \bigl( \overline\sigma_a^\rho(g_k) -   \overline\sigma_{n-a}^\rho(g_k) \bigr)\\
&\qquad - \sum_{(T_j, v_j) \to c_i^\Left} \biggl( \sum_{b+c=n-a} \tau_{abc}^\rho(T_j, v_j) -  \sum_{b+c=a} \tau_{(n-a)bc}^\rho(T_j, v_j) \biggr).
\end{align*}

Combining these two equations and summing over all closed leaves $c_1$, $c_2$, \dots, $c_s$ of the geodesic lamination $\lambda$, we obtain
\begin{align*}
&  \sum_{i=1}^s \sum_{(T_j, v_j) \to c_i} \biggl( \sum_{b+c=n-a} \tau_{abc}^\rho(T_j, v_j) -  \sum_{b+c=a} \tau_{(n-a)bc}^\rho(T_j, v_j) \biggr) \\
&\qquad \qquad \qquad \qquad \qquad \qquad \qquad \qquad \qquad +
 \sum_{i=1}^s  \sum_{g_k \to c_i} \bigl( \overline\sigma_a^\rho(g_k) -   \overline\sigma_{n-a}^\rho(g_k) \bigr) =0
\end{align*}
where the statement $g_k \to c_i$ is shorthand for `` $g_k \to c_i^\Right$ or $g_k \to c_i^\Left$ '', and similarly for $(T_j, v_j) \to c_i$.

Each infinite leaf $g_k$ contributes two terms $ \overline\sigma_a(g_k) -   \overline\sigma_{n-a}(g_k)$ to the second sum, respectively  equal to $\sigma_a^\rho(g_k) -  \sigma_{n-a}^\rho(g_k)$ for the positive end of $g_k$ and to $\sigma_{n-a}^\rho(g_k) -  \sigma_{a}^\rho(g_k)$ for the negative end. It follows that all terms in this second sum cancel out, so that we are only left with the first sum. 

A slightly different grouping of the terms of the first sum gives the equation of Proposition~\ref{prop:RelationsTriangleInvariants}. 
\end{proof}

A more conceptual and more general proof of Proposition~\ref{prop:RelationsTriangleInvariants}, using the length functions of \cite{Dre} and a cohomological argument, appears in \cite{BonDre}. 

We also prove in \cite{BonDre} that the relations of  Proposition~\ref{prop:RelationsTriangleInvariants} are the only constraints satisfied by the triangle invariants $\tau_{abc}(T_j, v_j)$. This property could also be proved with the results and techniques of the current article,  by elementary but somewhat cumbersome linear algebra. However, we prefer to omit it.

\end{document}